\documentclass[11pt]{amsart}
\usepackage[english]{babel} 
\usepackage[latin1]{inputenc} 
\usepackage{amssymb,geometry,color,graphicx}
\usepackage{amsmath,amsthm} 
\usepackage{booktabs}
\usepackage{datetime}
\usepackage{dsfont}  
\usepackage{amstext} 
\usepackage{graphicx}   
\usepackage{fancyhdr} 
\usepackage{latexsym}
\usepackage{mathrsfs}
\usepackage{listings}
\usepackage{cases}
\usepackage{mathtools}
\usepackage{bm}
\usepackage{caption}
\usepackage{subcaption}
\usepackage{tikz}
\usepackage{hyperref}
\usepackage{enumitem}
\hypersetup{
    colorlinks=true, %set true if you want colored links
    linktoc=all,     %set to all if you want both sections and subsections linked
    linkcolor=blue,  %choose some color if you want links to stand out
}

\usepackage{changes}
\definechangesauthor[name=Pawel, color=blue]{PP}

\numberwithin{equation}{section} 

\newcommand{\udef}{\mathrel{\mathop:}=}

\newcommand{\R}{\mathbb{R}}

\newcommand{\N}{\mathbb{N}}

\newcommand{\de}{\mathrm{d}}

\theoremstyle{plain}
\newtheorem{thm}{Theorem}[section]

\newtheorem{lem}[thm]{Lemma}

\newtheorem{rem}[thm]{Remark}
\newtheorem{exm}[thm]{Example}

    \title[On approximation of solutions of  SDDEs]{On approximation of solutions of stochastic delay differential equations via randomized Euler scheme}
\author[P. Przyby{\l}owicz]{Pawe{\l} Przyby{\l}owicz}
\address{AGH University of Krakow,
Faculty of Applied Mathematics,
Al. A.~Mickiewicza 30, 30-059 Krak\'ow, Poland}
\email{pprzybyl@agh.edu.pl, corresponding author}

\author[Y. Wu]{Yue Wu}
\address{Department of Mathematics and Statistics, University of Strathclyde, Glasgow, UK}
\email{yue.wu@strath.ac.uk}

\author[X. Xie]{Xinheng Xie}
\address{Department of Mathematics and Statistics, University of Strathclyde, Glasgow, UK}
\email{xinheng.xie@strath.ac.uk}
\begin{document}
\begin{abstract}
     We investigate existence, uniqueness and approximation of solutions to stochastic delay differential equations (SDDEs) under Carath\'eodory-type drift coefficients. Moreover, we also assume that both drift $f=f(t,x,z)$ and diffusion $g=g(t,x,z)$ coefficient are Lipschitz continuous with respect to the space variable $x$, but only H\"older continuous with respect to the delay variable $z$.  We provide a construction of randomized Euler scheme for approximation of solutions of Carath\'eodory SDDEs, and investigate its upper error bound.  Finally, we report results of numerical experiments that confirm our theoretical findings.
\end{abstract}
\maketitle

\section{Introduction}
In this paper, we investigate the efficiency of randomized numerical scheme for simulating the stochastic delay differential equations (SDDEs) by considering the SDDEs of the following form
\begin{equation}
\label{eq:DiscDDE2}
\begin{cases}
\mathrm{d}X(t)=f(t, X(t),X(t-\tau))\,\mathrm{d}t+g(t,X(t),X(t-\tau))\,\mathrm{d}W(t)\\
X(t)=x_0,  t\in[-\tau,0],
\end{cases}
\end{equation}
with the constant time-lag $\tau\in(0,+\infty)$, fixed time horizon $n\in\mathbb{N}$,  $f:[0,(n+1)\tau]\times\R^d\times\R^d\mapsto\R^d$,  $g:[0,(n+1)\tau]\times\R^d\times\R^d\mapsto\R^{d\times m}$, and $x_0\in\R^d$. We assume that the drift coefficient $f=f(t,x,z)$ is Borel measurable with respect to $t$, and (at least) continuous with respect to $(x,z)$. Therefore, the Carath\'eodory type conditions for $f$ are considered. For  the diffusion coefficient $g(t,x,z)$ we assume (at least) continuity with respect to all variables $(t,x,z)$.

Inspired by known Monte Carlo methods, randomized algorithms for approximation of stochastic integrals and solutions of stochastic differential equations (SDEs) have been recently considered in \cite{SHTD, SH1, RKYW2019, PPPM2014, PMPP2017, Morkisz2020,PSS2022}, to name but a few. The idea is to wisely combine the probabilistic representation for integrals with classical numerical schemes. A notable observation is that compared
to classical algorithms, their randomized counterparts may handle irregular coefficients well,
which relates to the flexible smoothness requirement of Monte Carlo methods. Suitably chosen randomization might be very helpful in order to handle time-irregularities in the right-hand side functions both for ODEs (\cite{BGMP2021, bochacik2,difonzo2022existence, hein_milla1,JentzenNeuenkirch,RKYW2017}) and SDEs (\cite{Morkisz2020,PMPP2017,PPPM2014,RKYW2019,PSS2022,Pss22}). 
 For instance, the classical Milstein method for non-autonomous SDEs requires the drift term to be differentiable with respect to  both temporal and spatial variables in order to achieve an order of convergence one. Its randomised version \cite{RKYW2019} can achieve the same order of convergence with only an Lipschitz condition (resp. a H\"older continuity) on the drift wrt the spatial variable (resp. the temporal variable).  

Though SDEs can be regarded as a special case of stochastic delay differential equations (SDDEs), the extention of the {\it randomized} numerical approaches from SDEs to SDDEs is non-trivial. On the one hand, compared to SDEs, the existence, uniqueness and $L^p$-H\"older regularity of the strong solution to \eqref{eq:DiscDDE2} is currently not known in the literature under  conditions for $f$ and $g$ considered in this paper. On the other hand, the time-delays may induce instabilities in the basic SDDEs \cite{mao2000delay}, and the presence of time-delays may further influence the convergence speed \cite{cao2015numerical}.
% An numerical example of a one-dimensional SDDE with smooth coefficients [Example 3.2, REF] shows that a couple of order-one methods for SDEs can only achieve an order of convergence half for approximating SDDEs. 
Both aspects bring challenges. 

In this article, we resolve these challenges one by one. For the study of the exact solution to \eqref{eq:DiscDDE2}, instead of considering the SDDE over the entire interval all at once, we follow a different way, known for CDDEs form in \cite{difonzo2022existence}. The entire time interval will be divided into multiple subintervals of the length $\tau$ and the solution of the SDDE will be considered at each subinterval separately. This will allow the SDDE to be converted into a sequence of iterative SDEs with random coefficients and analyzed by induction, where the delay term is treated as a random resource that has been given (see \eqref{eq:dynamics}). The randomized Euler-Maruyama method is defined in the same manner, i.e., iteratively. We keep the same grid for all the subintervals of length $\tau$ so that the simulation obtained from the preceding subinterval will directly be the delay input for the current subinterval. To assist the error analysis, we introduce the auxiliary randomized Euler-Maruyama scheme.
%, whose filtration will be different from the one of the randomized Euler-Maruyama method. 
In the case when $f$ or $g$ are H\"older contiuous with respect to $z$ we observe that the numerical error accumulates over the subintervals. In particular, this may suggest that this analysis is not valid for the infinite time horizon cases.

To summarise, the main contributions of the paper are as follows:
\begin{itemize}
    \item We show existence, uniqueness, and H\"older regularity of the strong solution to \eqref{eq:DiscDDE2} when both the drift $f(t,x,z)$ and diffusion $g(t,x,z)$ coefficients are Borel measurable with respect to $t$, satisfy a global Lipschitz condition  with respect to $x$ andf global linear growth condition wrt $(x,z)$, but are  continuous with respect to the delay variable $z$  (see Theorem \ref{sol_dde_prop_1}).
    \item We perform rigorous error analysis of the randomized Euler scheme applied to \eqref{eq:DiscDDE2} when both  drift coefficient $f(t,x,z)$ and diffusion coefficient $g(t,x,z)$ satisfy a global Lipschitz condition with respect to $x$ and a global H\"older condition with respect to $z$. Still, we assume that $f$ is only Borel measurable with respect to $t$, but for $g$ we assume now that it is also H\"older contiuous with respect to the time variable $t$.  (Theorem \ref{rate_of_conv_expl_Eul}).
    \item We present an implementation of the randomized Euler-Maruyama method in Python code and report results of numerical experiments that show stable error behaviour as stated in Theorem \ref{rate_of_conv_expl_Eul}.
\end{itemize}

The structure of the article is as follows. Basic notions, definitions together with assumptions and the construction of the randomized Euler-Maruyama scheme are given in Section 2. All Section 3 is devoted to the issue of existence and uniqueness of strong solutions of the  Carath\'eodory type SDDEs \eqref{eq:DiscDDE2}.  Section 4 contains proof of the main result of the paper (Theorem \ref{sol_dde_prop_1}) that states upper bounds on the error of the randomized Euler-Maruyama  scheme. In Section 5 we report results of numerical experiments with an exemplary Python implementation.

\section{Preliminaries}
Define $\mathbb{N}=\{1,2,\ldots\}$,  $\mathbb{N}_0=\mathbb{N}\cup \{0\}$. For an integer $k$, $[k]:=\{1,\ldots,k\}$ and $[k]_0:=\{0\}\cup [k]$. By $|\cdot|$ we mean the Euclidean norm in $\R^d$ or the Frobenius norm in $\mathbb{R}^{d\times m}$.
We consider a  complete probability space $(\Omega,\Sigma,\mathbb{P})$. For a random variable $X:\Omega\to\mathbb{R}^d$ we denote by $\|X\|_{L^p(\Omega)}=(\mathbb{E}|X|^p)^{1/p}$, where $p\in [2,+\infty)$. We denote by $(\Sigma_t)_{t\geq 0}$ a filtration, satisfying the usual conditions, such that $W=(W(t))_{t\geq 0}$ is $m$-dimensional Wiener process on $(\Omega,\Sigma,\mathbb{P})$ with respect to $(\Sigma_t)_{t\geq 0}$. Let $\Sigma_{\infty}=\sigma\Bigl(\bigcup_{t\geq 0}\Sigma_t\Bigr)$. 
 For two sub-$\sigma$-fields $\mathcal{A}$, $\mathcal{B}$ of $\Sigma$ we denote by $\mathcal{A}\vee \mathcal{B}=\sigma(\mathcal{A}\cup\mathcal{B})$.

Let us fix the {\it horizon parameter} $n\in\mathbb{N}$. On the drift coefficient $f:[0,(n+1)\tau]\times\R^d\times\R^d\mapsto\R^d$ we impose the following assumptions:
\begin{enumerate}[label=\textbf{(A\arabic*)},ref=(A\arabic*)]
    \item\label{ass:A1} $f(t,\cdot,\cdot)\in C(\R^d\times\R^d;\R^d)$ for all $t\in [0,(n+1)\tau]$,
    \item\label{ass:A2} $f(\cdot,x,z):[0,(n+1)\tau]\to\mathbb{R}^d$ is Borel measurable for all $(x,z)\in \R^d\times\R^d$,
    \item\label{ass:A4}
        There exist $K_f\in (0,\infty)$ such that  for all $t\in [0,(n+1)\tau]$, $x,x_1,x_2,z\in\R^d$
    \begin{align}\label{assumpt_a4}
        \begin{split}
        &|f(t,x,z)|\leq K_f (1+|x|+|z|),\\
          & |f(t,x_1,z)-f(t,x_2,z)|\leq K_f|x_1-x_2|.
        \end{split}
    \end{align}
\end{enumerate}
For the diffusion coefficient $g:[0,(n+1)\tau]\times\R^d\times\R^d\mapsto\R^{d\times m}$ we impose the following assumptions:
\begin{enumerate}[label=\textbf{(B\arabic*)},ref=(B\arabic*)]
    \item\label{ass:B1} $g(t,\cdot,\cdot)\in C(\R^d\times\R^d;\R^{d\times m})$ for all $t\in [0,(n+1)\tau]$,
    \item\label{ass:B2} $g(\cdot,x,z):[0,(n+1)\tau]\to\mathbb{R}^{d\times m}$  is Borel measurable for all $(x,z)\in \R^d\times\R^d$,
    \item\label{ass:B3}
       There exists $K_g\in (0,\infty)$ such that  for all $t\in [0,(n+1)\tau]$, $x,x_1,x_2,z\in\R^d$
    \begin{align}\label{assumpt_b3}
        \begin{split}
        &|g(t,x,z)|\leq K_g (1+|x|+|z|),\\
          & |g(t,x_1,z)-g(t,x_2,z)|\leq K_g|x_1-x_2|.
        \end{split}
    \end{align}
\end{enumerate}
In Section 3 we show that,  under the assumptions \ref{ass:A1}-\ref{ass:A4}, \ref{ass:B1}-\ref{ass:B3}, the SDDE \eqref{eq:DiscDDE2} has an unique strong solution. Next, in Section 4 we investigate error of the {\it randomized Euler scheme} under slightly stronger assumptions than  \ref{ass:A1}-\ref{ass:A4}, \ref{ass:B1}-\ref{ass:B3}.
\begin{table}
\begin{center}
\begin{tabular}{ c|c} 
 \hline
 Symbol & Meaning  \\ 
   \hline
     $\gamma_k^j$ & uniformly distributed sample from $[0,1]$ for $j\in \mathbb{N}_0$ and $k\in \mathbb{N}$\\
 $t_k^j$& $t_k^j=j\tau+kh$ for $k\in [N]_0$ and $j\in [n]_0$ \\
 $\theta_{k+1}^j$ & $\theta_{k+1}^j=t_k^j+h\gamma_{k+1}^j$\\
$\delta_{k+1}^{j}$ & $\delta_{k+1}^{j}=kh+h\gamma_{k+1}^{j}$\\
   \hline
\end{tabular}
\caption{Notations.}
\label{table:1}
\end{center}
\end{table}
The aforementioned randomized Euler scheme is defined as follows. Fix the {\it discretization parameter} $N\in\N$ and set
\begin{displaymath}
	t_k^j=j\tau+kh, \quad k\in [N]_0, \ j\in [n]_0,
\end{displaymath} 
 where
\begin{equation}
	h=\frac{\tau}{N}.
\end{equation}
Note that for each $j$ the sequence $\{t^j_k\}_{k=0}^N$ provides uniform discretization of the subinterval $[j\tau,(j+1)\tau]$.
Let $\{\gamma_k^j\}_{j\in\mathbb{N}_0,k\in\mathbb{N}}$ be an iid sequence of random variables, defined on the complete probability space $(\Omega,\Sigma,\mathbb{P})$, where every $\gamma_k^j$ is uniformly distributed on $[0,1]$. We assume that the $\sigma$-fields $\sigma(\{\gamma_k^j\}_{j\in\mathbb{N}_0,k\in\mathbb{N}})$ and $\Sigma_{\infty}$ are independent. Then $W=(W(t))_{t\geq 0}$ is also the Wiener process with respect to the extended filtration 
\begin{equation}
    \tilde\Sigma_t=\Sigma_t\vee\sigma(\{\gamma_k^j\}_{j\in\mathbb{N}_0,k\in\mathbb{N}}), \quad t\geq 0.
\end{equation}
We set $y_0^{-1}=\ldots=y^{-1}_N=x_0$ and then  for $j=[n]_0$, $k=[N-1]_0$ we take
\begin{eqnarray}
\label{expl_euler_1}
	&&y_0^j=y^{j-1}_N,\\
\label{expl_euler_11}	
	&&y_{k+1}^j= y_k^j+h\cdot f(\theta_{k+1}^j,y_k^j,y_k^{j-1})+g(t_k^j,y_k^j,y_k^{j-1})(W(t_{k+1}^j)-W(t_{k}^j)),
\end{eqnarray}
where $\theta_{k+1}^j=t_k^j+h\gamma_{k+1}^j$. As the output we obtain the  sequence of $\mathbb{R}^d$-valued random vectors  $\{y_k^j\}_{k\in [N]_0,j\in [n]_0}$ that provides a discrete approximation of the values $\{X(t_k^{j})\}_{k\in [N]_0, j\in [n]_0}$. By induction we get that  for all $j\in [n]_0$, $k\in [N]_0$
\begin{eqnarray}
\label{sig_filed2}
&&\sigma(\{y_k^j\})\subset\sigma\Bigl(\{\gamma_1^0\ldots,\gamma_N^0,\ldots,\gamma_1^{j-1},\ldots,\gamma_N^{j-1},\gamma_1^j,\ldots,\gamma_k^j\}\Bigr)\notag\\
&&\quad\quad\vee \sigma\Bigl(\{W(t_0^0),\ldots, W(t_N^0),\ldots,W(t_0^j),\ldots, W(t_k^j)\}\Bigr).
\end{eqnarray}
As the horizon parameter $n$ is fixed, the randomized Euler scheme uses $O(N)$ evaluations of $f$ (with a constant
in the '$O$' notation that depends on $n$ but not on $N$).

The aim is to establish upper bounds on the $L^p(\Omega)$-error
\begin{equation}\label{eq:L2err}
    \Bigl\|\max\limits_{0\leq k\leq N}|X(t_k^j)-y_k^j|\Bigl\|_{L^p(\Omega)}
\end{equation}
for all $j=0,1,\ldots,n$.
%%%%%%%%%%
\section{Properties of solutions to Carath\'eodory SDDEs}

 We take into account the semi-flow property holds for SDDE \eqref{eq:DiscDDE2}.  Namely, we  define 
 \begin{equation}
 \label{def_phi_l}
    \phi_{l}(t):=X(t+l\tau), \quad t\in [0,\tau],\quad      l=-1,0,1,\ldots,n
 \end{equation}
and, in particular,  $\phi_{-1}(t)=x_0$ for $t\in [0,\tau]$. Using change of variable formula for Lebesgue and It\^o integrals we arrive at the following SDE with random coefficients
\begin{equation}
\label{eq:dynamics}
\begin{cases}
\mathrm{d}\phi_{l}(t)=f_{l}\big(t, \phi_{l}(t)\big)\,\mathrm{d}t+g_{l}\big(t,\phi_{l}(t)\big)\,\mathrm{d}W_{l}(t), \quad t\in[0,\tau], \\
\phi_{l}(0)=\phi_{l-1}(\tau),
\end{cases}
\end{equation}
where the random fields $f_{l}:\mathbb{R}\times \mathbb{R}^d \times \Omega \to \mathbb{R}^d$ and $g_{l}:\mathbb{R}\times \mathbb{R}^d \times \Omega \to \mathbb{R}^{d\times m}$ are defined follows 
\begin{eqnarray}
\label{def_fgl}
    &&f_{l}(t, x):=f\big(t+l\tau, x,\phi_{l-1}(t)\big),\\
    &&g_{l}(t, x):=g\big(t+l\tau, x,\phi_{l-1}(t)\big),
\end{eqnarray}
 and $W_{l}(t):=W(t+l\tau)$, $t\in [0,\tau]$, $x\in\mathbb{R}^d$ \footnote{We can also take $W_{l}(t):=W(t+l\tau)-W(l\tau)$ - in both cases $W_l$ is a Wiener process with respect to the filtration $(\Sigma_{t}^l)_{t\in [0,\tau]}$= $(\Sigma_{t+l\tau})_{t\in [0,\tau]}$, see, for example, Theorem 100, page 67 in \cite{Situ}.}. We also set 
 \begin{eqnarray}
  &&\Sigma^{-1}_t:=\{\emptyset,\Omega\},\\
  &&\Sigma^j_t:=\Sigma_{t+j\tau},
 \end{eqnarray}
 for $t\in [0,\tau]$, $j\in [n]_0$. The solution $X$ of \eqref{eq:DiscDDE2} can be written as
\begin{equation}
    X(t)=\sum\limits_{j=-1}^n\phi_j(t-j\tau)\cdot\mathbf{1}_{[j\tau,(j+1)\tau]}(t), \quad t\in [-\tau,(n+1)\tau].
\end{equation}
We state $L^p$-boundedness and regularity of the solution $X$ in the following Theorem \ref{sol_dde_prop_1}. Similar result is given in Theorem 3.1. at pages 156-157 in \cite{XMao}, however, the $L^p$ estimates and $L^p$-H\"older regularity of the solution is not studied there.
%%%%%%%%%%%%%%%%
\begin{thm}
\label{sol_dde_prop_1}
    Let $n\in\N_0$, $\tau\in (0,+\infty)$, $x_0\in\R^d$ and let $f,g$ satisfy the assumptions \ref{ass:A1}-\ref{ass:A4}, \ref{ass:B1}-\ref{ass:B3}. Then there exists a unique strong solution $X=X(x_0,f,g)$ to \eqref{eq:DiscDDE2}
    such that for $j\in [n]_0$ we have
    \begin{equation}
    \label{upper_est_Kj}
        \mathbb{E}\Bigl(\sup\limits_{0\leq t \leq \tau}|\phi_j(t)|^p\Bigr)\leq K_j,
    \end{equation}
    where  $K_{-1}:=|x_0|^p$, $K=\max\{K_f,K_g\}$,
    \begin{equation}
    \label{def_K_j}
        K_j= C_p\Bigl(K_{j-1}+c_p\tau^{p/2}K^p(\tau^{p/2}+1)(1+K_{j-1})\Bigr)\exp\Bigl(C_p\tau^pK^p(1+K^p)\Bigr),
    \end{equation}
    and for all $j\in [n]_0$, $t,s\in [0,\tau]$ it holds
    \begin{equation}
    \label{phi_j_Lipschitz}
        \|\phi_j(t)-\phi_j(s)\|_{L^p(\Omega)}\leq c_pK(\tau^{1/2}+1)(1+K_{j-1}^{1/p}+K_{j}^{1/p})|t-s|^{1/2}.
    \end{equation}
\end{thm}
\begin{proof} We proceed by induction.    We start with the case when $j=0$ and consider the following SDE
\begin{equation}
\label{eqODE_0}
    \begin{cases}
        \mathrm{d}\phi_0(t)=f_0(t,\phi_0(t))\,\mathrm{d}t+g_0(t,\phi_0(t))\,\mathrm{d}W_0(t), & t\in [0,\tau], \\
        \phi_0(0)=x_{0}, 
    \end{cases}
\end{equation}
with $f_0(t,x)=f(t,x,\phi_{-1}(t))=f(t,x,x_0)$, $g_0(t,x)=f(t,x,\phi_{-1}(t))=g(t,x,x_0)$. Moreover, by \eqref{assumpt_a4}, \eqref{assumpt_b3} we have for all $t\in [0,\tau]$, $x,y\in\mathbb{R}^d$ that
\begin{equation}
\max\{|f_0(t,x)|,|g_0(t,x)|\}\leq \max\{K_f,K_g\}(1+|x_0|+|x|),
\end{equation} 
and
\begin{equation}
    \max\{|f_0(t,x)-f_0(t,y)|,|g_0(t,x)-g_0(t,y)|\}\leq \max\{K_f,K_g\} |x-y|.
\end{equation}
Therefore, by Proposition 3.28, page 187 in \cite{PARRAS} we have that there exists a unique strong solution $\phi_0:[0,\tau]\times\Omega\to\R^d$ of \eqref{eqODE_0} that is adapted to $(\Sigma^0_t)_{t\in [0,\tau]}$. By Remark 3.29, page 188 in \cite{PARRAS}, the solution $\phi_0$ satisfies \eqref{upper_est_Kj} with $j=0$ and
\begin{equation}
    K_0=C_p\Bigl[|x_0|^p+c_p \tau^{p/2} K^p (1+\tau^{p/2}) (1+|x_0|^p)\Bigr]\exp\Bigl(C_p\tau^pK^p(1+K^p)\Bigr),
\end{equation}
with $K=\max\{K_f,K_g\}$. In addition, 
%if $K\in L^p([0,(n+1)\tau])$ for some $p\in (1,+\infty]$ then by Lemma \ref{lem_ode_1} 
$\phi_0$ satisfies \eqref{phi_j_Lipschitz} for $j=0$. 

Let us now assume that for some $j\in [n-1]_0$ there exists, adapted to $(\Sigma^j_t)_{t\in [0,\tau]}$,  a unique strong solution $\phi_j:[0,\tau]\times\Omega\to\R^d$ of
\begin{equation}
\label{eqODE_j}
\begin{cases}
\mathrm{d}\phi_j(t)=f_j(t,\phi_j(t))\,\mathrm{d}t+g_j(t,\phi_j(t))\,\mathrm{d}W_j(t), & t\in [0,\tau], \\
\phi_j(0)=\phi_{j-1}(\tau), 
\end{cases}
\end{equation}
that satisfies \eqref{upper_est_Kj}, \eqref{phi_j_Lipschitz}, where $f_j(t,x)=f(t+j\tau,x,\phi_{j-1}(t))$, $g_j(t,x)=g(t+j\tau,x,\phi_{j-1}(t))$. We consider the following SDE
\begin{equation}
\label{eqODE_j_1}
\begin{cases}
\mathrm{d}\phi_{j+1}(t)=f_{j+1}(t,\phi_{j+1}(t))\,\mathrm{d}t+g_{j+1}(t,\phi_{j+1}(t))\,\mathrm{d}W_{j+1}(t), & t\in [0,\tau], \\
\phi_{j+1}(0)=\phi_j({\tau}), 
\end{cases}
\end{equation}
with $f_{j+1}(t,x)=f(t+(j+1)\tau,x,\phi_j(t))$, $g_{j+1}(t,x)=g(t+(j+1)\tau,x,\phi_j(t))$. Since the process $\phi_j$ is adapted to $(\Sigma^j_t)_{t\in [0,\tau]}$, $\Sigma^j_t\subset\Sigma^{j+1}_t$ and has continuous trajectories,  for all $x\in\R^d$ the processes $(f_{j+1}(t,x))_{t\in [0,\tau]}$, $(g_{j+1}(t,x))_{t\in [0,\tau]}$ are $(\Sigma^{j+1}_t)_{t\in [0,\tau]}$-progressively measurable. Moreover, by \eqref{assumpt_a4}, \eqref{assumpt_b3} for all $(t,x)\in [0,\tau]\times\R^d$
\begin{equation}
    \max\{|f_{j+1}(t,x)|,|g_{j+1}(t,x)|\}\leq K (1+|\phi_j(t)|)+K|x|,
\end{equation}
and  for all $t\in [0,\tau]$, $x,y\in \mathbb{R}^d$ it holds
\begin{equation}
    \max\{|f_{j+1}(t,x)-f_{j+1}(t,y)|, |g_{j+1}(t,x)-g_{j+1}(t,y)|\}\leq K|x-y|.
\end{equation}
Hence, by  by Proposition 3.28, page 187 in \cite{PARRAS} there exists a unique strong solution solution $\phi_{j+1}:[0,\tau]\times\Omega\to\R^d$ of \eqref{eqODE_j_1}. By the inductive assumption and Remark 3.29, page 188 in \cite{PARRAS} we get that
\begin{eqnarray}
&&\mathbb{E}\Bigl(\sup\limits_{0\leq t\leq \tau}|\phi_{j+1}(t)|^p\Bigr)\\&&\leq C_p\Biggl(\mathbb{E}|\phi_j(\tau)|^p+\mathbb{E}\Bigl(\int\limits_0^{\tau}K(1+|\phi_j(s)|)\,\mathrm{d}s\Bigr)^p+\mathbb{E}\Bigl(\int\limits_0^{\tau}K^2(1+|\phi_j(s)|)^{p/2}\,\mathrm{d}s\Bigr)^p\Biggr)\notag\\
&&\qquad \times\exp\Bigl(C_p\tau^{p-1}\int\limits_0^{\tau}(K^p+K^{2p})\,\mathrm{d}s\Bigr)\notag\\
&&\leq C_p\Bigl(K_j+c_p\tau^{p/2}K^p(\tau^{p/2}+1)(1+K_j)\Bigr)\exp\Bigl(C_p\tau^pK^p(1+K^p)\Bigr)=K_{j+1},\notag
\end{eqnarray}
and therefore, by the H\"older and Burkholder inequalities, we get for $s,t\in [0,\tau]$, $s<t$
\begin{eqnarray}
    &&\mathbb{E}|\phi_{j+1}(t)-\phi_{j+1}(s)|^p\leq C_p\Bigl[(t-s)^{p-1}\int\limits_s^t\mathbb{E}|f_{j+1}(u,\phi_{j+1}(u))|^p\,\mathrm{d}u\notag\\
    &&+\hat C_p(t-s)^{\frac{p-2}{2}}\int\limits_s^t\mathbb{E}|g_{j+1}(u,\phi_{j+1}(u))|^p\,\mathrm{d}u\Bigr]\leq \bar K_{j+1}(t-s)^{p/2},
\end{eqnarray}
where $\bar K_{j+1}=c_p^p K^p(\tau^{p/2}+1)(1+K_j+K_{j+1})$. This ends the inductive proof.
\end{proof}
%%%%%%%%%%%%%%%%%%%%
\section{Error analysis for the randomized Euler algorithm}\label{sec:alpha}
In this section we provide error analysis for randomized Euler scheme under slightly stronger assumptions than those stated in Section 2. Namely,  for the drift coefficient we assume that instead of \ref{ass:A4} it  satisfies what follows:
\begin{itemize}
    \item [(A3')]
        There exist $\alpha_1\in (0,1]$, $\bar K_f, L_f\in (0,\infty)$ such that  for all $t\in [0,(n+1)\tau]$
        \begin{equation}
            |f(t,0,0)|\leq \bar K_f,
        \end{equation}
        and for all $t\in [0,(n+1)\tau]$, $x_1,x_2,z_1,z_2\in \mathbb{R}^d$
    \begin{equation}
    \label{assumpt_a31}
         |f(t,x_1,z_1)-f(t,x_2,z_2)|\leq L_f (|x_1-x_2|+|z_1-z_2|^{\alpha_1}).
    \end{equation}
\end{itemize}    
For the diffusion coefficient we impose the following assumption, that is stronger than \ref{ass:B1}, \ref{ass:B2}, and \ref{ass:B3}:
\begin{itemize}
    \item [(B3')]
        There exist $\alpha_2,\varrho\in (0,1]$, $\bar K_g, L_g\in (0,\infty)$ such that  for all $t\in [0,(n+1)\tau]$, $x,z\in\mathbb{R}^d$
        \begin{equation}
        \label{assumpt_b31}
            |g(t_1,x,z)-g(t_2,x,z)|\leq \bar K_g(1+|x|+|z|)|t_1-t_2|^{\varrho},
        \end{equation}
        and for all $t\in [0,(n+1)\tau]$, $x_1,x_2,z_1,z_2\in \mathbb{R}^d$
    \begin{equation}
    \label{assumpt_b32}
         |g(t,x_1,z_1)-g(t,x_2,z_2)|\leq L_g (|x_1-x_2|+|z_1-z_2|^{\alpha_2}).
    \end{equation}
\end{itemize}   
Note that for any $f$ that satisfies the assumption (A3') it holds for all $t\in [0,(n+1)\tau]$, $x,z\in\mathbb{R}^d$
\begin{equation}
    |f(t,x,z)|\leq K_f(1+|x|+|z|),
\end{equation}
with $K_f=\bar K_f+L_f$, while for any $g$ satisfying (B3') we have for all $t\in [0,(n+1)\tau]$, $x,z\in\mathbb{R}^d$
\begin{equation}
    |g(t,x,z)|\leq K_g(1+|x|+|z|),
\end{equation}
with $K_g=|g(0,0,0)|+L_g+\bar K_g((n+1)\tau)^{\varrho}$.

In order to perform error analysis, we define the auxiliary randomized Euler-Maruyama scheme as follows:
    \begin{eqnarray}
        \label{aux_euler_1}
	        &&\bar y_0^0= y^0_0=y^{-1}_N=x_0,\\
        \label{aux_euler_11}	
	        &&\bar y_{k+1}^0=\bar y_k^0+h\cdot f_0(\theta_{k+1}^0,\bar y_k^0)+g_0(t_k^0, \bar y_k^0)\Delta_k^0 W, \ k=0,1,\ldots, N-1,
    \end{eqnarray}
    and for $l=0,1,\ldots,n-1$
    \begin{eqnarray}
        \label{aux_euler_l1}
	       &&\bar y_0^{l+1}= y^{l+1}_0=y^{l}_N,\\
        \label{aux_euler_l11}	
	        &&\bar y_{k+1}^{l+1}=\bar y_k^{l+1}+h\cdot f_{l+1}(\delta_{k+1}^{l+1},\bar y_k^{l+1})+g_{l+1}(t_k^{0}, \bar y_k^{l+1})\Delta_k^{l+1} W, \ k=0,1,\ldots, N-1.
\end{eqnarray}
where  $\delta_{k+1}^{l+1}=kh+\gamma_{k+1}^{l+1}h$ and $\Delta_k^{l+1} W:=W(t_{k+1}^{l+1})-W(t_{k}^{l+1})$ for $l=-1, 0,\ldots,n-1$. 

Note that  $\delta_{k+1}^{l+1}$ is uniformly distributed in $(t_k^0,t_{k+1}^0)$ and $\theta_{k+1}^{l+1}=\delta_{k+1}^{l+1}+(l+1)\tau$ is uniformly distributed in $(t_k^{l+1},t_{k+1}^{l+1})$. For $N\geq 2$ we have that $h\in (0,\tau)$, which gives that $t_k^{l+1}>t_{k+1}^l$ and $\Sigma_{t_{k+1}^l}\subset \Sigma_{t_{k}^{l+1}}$. This fact and Lemma \ref{meas_Xgamma} imply that for $\phi_l(\delta_{k+1}^{l+1})=X(t_k^l+h\gamma_{k+1}^{l+1})$, where $\phi_l$ is defined in \eqref{def_phi_l}, we have 
\begin{equation}
\label{meas_phi_l_d}
\sigma(\phi_l(\delta_{k+1}^{l+1}))\subset\sigma(\{\gamma_{k+1}^{l+1}\})\vee\Sigma_{t_{k+1}^l}\subset\sigma(\{\gamma_{k+1}^{l+1}\})\vee\Sigma_{t_{k}^{l+1}},    
\end{equation}
and hence, $\phi_l(\delta_{k+1}^{l+1})$ is independent of $\Delta_k^{l+1}W$. Moreover, it follows from \eqref{def_fgl}, \eqref{aux_euler_l11} that 
\begin{eqnarray}
\label{aux_E_ykl_ind}
    &&\bar y_{k+1}^{l+1}=\bar y_k^{l+1}+h\cdot f(t_k^{l+1}+h\gamma_{k+1}^{l+1},\bar y_k^{l+1},\phi_l(\delta_{k+1}^{l+1}))\notag\\
    &&\quad\quad+ g(t_k^{l+1},\bar y_k^{l+1},\phi_l(t_k^{l+1}))\cdot\Delta_k^{l+1}W.
\end{eqnarray}
Hence, by \eqref{aux_E_ykl_ind}, \eqref{meas_phi_l_d}, and induction we get for $j\in [n]_0$ and $k\in [N]_0$ that 
\begin{equation}
\label{meas_bar_ykj}
    \sigma(\bar y_k^j)\subset \Sigma_{t_k^j}\vee\sigma(\{\gamma_1^0\ldots,\gamma_N^0,\ldots,\gamma_1^{j-1},\ldots,\gamma_N^{j-1},\gamma_1^j,\ldots,\gamma_k^j\}).
\end{equation}
Therefore, $\bar y_k^j$ is measurable with respect to larger $\sigma$-field than $y_k^j$ (see \eqref{sig_filed2}).
Note that $\bar y^{l}_k$  is not implementable. However, we use $\bar y^{l}_k$ only in order to estimate the error \eqref{eq:L2err} of (implementable scheme) $y^{l}_k$.
%%%%%%%
\begin{thm} 
\label{rate_of_conv_expl_Eul} 
Let $n\in\N_0$, $\tau\in (0,+\infty)$, $x_0\in\R^d$, and let $f,g$ satisfy the assumptions (A1), (A2), (A3') and (B1), (B2), (B3') for some $\varrho, \alpha_1, \alpha_2\in (0,1]$. There exist $C_0,C_1,\ldots,C_n\in (0,+\infty)$ such that for  all $N\geq \lceil \tau\rceil$ and $j=0,1,\ldots,n$ it holds
	\begin{equation}
	\label{error_main_thm}
		\Bigl\|\max\limits_{0\leq k\leq N}|X(t_k^j)-y_k^j|\Bigl\|_{L^p(\Omega)}\leq C_j h^{\min\{\varrho,\frac{1}{2}\}\alpha^{j}},
	\end{equation}
	where $\alpha:=\min\{\alpha_1,\alpha_2\}$, and, in particular, if $\alpha=1$ then
 \begin{equation}
	\label{sc_error_main_thm}
		\Bigl\|\max\limits_{0\leq k\leq N}|X(t_k^j)-y_k^j|\Bigl\|_{L^p(\Omega)}\leq C_j h^{\min\{\varrho,\frac{1}{2}\}}.
	\end{equation}
\end{thm}
\begin{proof} 
    We start with the initial-vale problem \eqref{eq:dynamics} with $l=0$. Since $\bar y_0^0=y_0^0=x_0$, $f_0(t,x)=f(t,x,x_0)$ and $g_0(t,x)=g(t,x,x_0)$, we have that $\bar y_k^0=y_k^0$ for all $k=0,\ldots,N$. Moreover, by (A1), (A2), (A3') and (B1), (B2), (B3')  we have that $f_0$ and $g_0$ is Borel measurable, and for all $t, t_1, t_2\in [0,\tau]$ and  $x,y\in\R^d$ 
    \begin{eqnarray}
        &&|f_0(t,x)|\leq K_f(1+|x_0|+|x|),\notag\\
        &&|f_0(t,x)-f_0(t,y)|\leq L_f|x-y|.
    \end{eqnarray}
    and
     \begin{eqnarray}
       &&|g_0(t,x)|\leq  K_g(1+|x_0|+|x|),\notag\\
       &&|g_0(t,x)-g_0(t,y)|\leq K_g|x-y|\\
      &&|g_0(t_1,x)-g_0(t_2,x)|\leq \bar K_g(1+|x|+|x_0|)|t_1-t_2|^{\varrho}. \notag
    \end{eqnarray}
    Since for $l=0$ we deal with ordinary SDE, by Theorem \ref{sol_dde_prop_1} and by using analogous arguments as in the proof of Proposition 1 in \cite{PMPP2017} we get
\begin{eqnarray}
   && \Bigl\|\max\limits_{0\leq k\leq N}|\phi_0(t_k^0)-y_k^0|\Bigl\|_{L^p(\Omega)}\leq C_0h^{\min\{\varrho,\frac{1}{2}\}},\notag
\end{eqnarray}
where $C_0$ does not depend on $N$. Since $\phi_0(t_k^0)=X(t_k^0)$ we get \eqref{error_main_thm} for $j=0$.

Let us now assume that there exists $l\in\{0,1,\ldots,n-1\}$ for which there exists $C_l \in (0,+\infty)$ such that for all $N\geq \lceil \tau\rceil$ 
\begin{equation}
\label{ind_assumpt_1}
    \Bigl\|\max\limits_{0\leq k\leq N}|\phi_l(t_k^0)-y_k^l|\Bigl\|_{L^p(\Omega)}\leq C_l h^{\min\{\varrho,\frac{1}{2}\}\alpha^{l}}.
\end{equation}
We consider the initial-value problem \eqref{eq:dynamics} for $\phi_{l+1}$. By (A1), (A2), (A3') and (B1), (B2), (B3') we have that $f_{l+1}$ and $g_{l+1}$ are Borel measurable, and for all $t, t_1, t_2 \in [0,\tau]$,  $x,y\in\R^d$ 
    \begin{eqnarray}
    \label{eqn:fboundl+1}
        &&|f_{l+1}(t,x)|\leq K_f(1+|x|+|\phi_l(t)|),\notag\\
        &&|f_{l+1}(t,x)-f_{l+1}(t,y)|\leq L_f|x-y|.
    \end{eqnarray}
    and 
     \begin{eqnarray}
     \label{eqn:gboundl+2}
     &&|g_{l+1}(t,x)|\leq K_g(1+|x|+|\phi_l(t)|),\notag\\
        &&|g_{l+1}(t,x)-g_{l+1}(t,y)|\leq L_g|x-y|\\     &&|g_{l+1}(t_1,x)-g_{l+1}(t_2,x)|\leq \bar K_g (1+|x|+|\phi_l(t_1)|)|t_1-t_2|^{\varrho}\notag\\
        &&\quad\quad+L_g|\phi_l(t_1)-\phi_l(t_2)|^{\alpha_2}. \notag
    \end{eqnarray}
The following error decomposition holds
\begin{equation}\label{error_decomp}
\begin{aligned}
\Bigl\|\max\limits_{0\leq k\leq N}|\phi_{l+1}(t_k^0)-y_k^{l+1}|\Bigl\|_{L^p(\Omega)} &\leq \Bigl\|\max\limits_{0\leq k\leq N}|\phi_{l+1}(t_k^0)-\bar y_k^{l+1}|\Bigl\|_{L^p(\Omega)} \\
&\quad+\Bigl\|\max\limits_{0\leq k\leq N}|\bar y_k^{l+1}-y_k^{l+1}|\Bigl\|_{L^p(\Omega)}.
\end{aligned}
\end{equation}
Firstly, we estimate $\displaystyle{\Bigl\|\max\limits_{0\leq k\leq N}|\bar y_k^{l+1}-y_k^{l+1}|\Bigl\|_{L^p(\Omega)}}$. For $k\in [N]$ we get
\begin{align*}
\bar y^{l+1}_k-y^{l+1}_k &= \sum\limits_{j=1}^k(\bar y^{l+1}_j-\bar y^{l+1}_{j-1})-\sum\limits_{j=1}^k(y^{l+1}_j-y^{l+1}_{j-1})\notag\\
&=h\sum\limits_{j=1}^k\Bigl(f(\theta_j^{l+1},\bar y^{l+1}_{j-1},\phi_{l}(\delta_{j}^{l+1}))-f(\theta_j^{l+1},y^{l+1}_{j-1},y^l_{j-1})\Bigr)\\
&\quad+\sum\limits_{j=1}^k \Bigl(g(t_{j-1}^{l+1},\bar y^{l+1}_{j-1},\phi_{l}(t_{j-1}^{0}))-g(t_{j-1}^{l+1},y^{l+1}_{j-1},y^l_{j-1})\Bigr)\Delta_{j-1}^{l+1}W,
\end{align*}
where  $\Delta_{j-1}^{l+1}W=W(t_j^{l+1})-W(t_{j-1}^{l+1})$. This gives by the assumption \eqref{assumpt_b31}  that
\begin{align*}
   |\bar y^{l+1}_k-y^{l+1}_k|^p&\leq  C_pL_f^p \Big(h\sum\limits_{j=1}^k|\bar y^{l+1}_{j-1}-y^{l+1}_{j-1}|+h\sum\limits_{j=1}^k  |\phi_{l}(\delta_{j}^{l+1})-y^{l}_{j-1}|^{\alpha_1}\Big)^p\\
   &\quad +C_p\Big|\sum\limits_{j=1}^k \Bigl(g(t_{j-1}^{l+1},\bar y^{l+1}_{j-1},\phi_{l}(t_{j-1}^{0}))-g(t_{j-1}^{l+1},y^{l+1}_{j-1},y^l_{j-1})\Bigr)\Delta_{j-1}^{l+1}W\Big|^p\\
 %  &=:T^k_1+T^k_2,
\end{align*}
Let us denote by 
\begin{equation}
\label{def_Glj}
G_{j-1}^{l+1}=g(t_{j-1}^{l+1},\bar y^{l+1}_{j-1},\phi_{l}(t_{j-1}^{0}))-g(t_{j-1}^{l+1},y^{l+1}_{j-1},y^l_{j-1}).
\end{equation}
 Hence, for $k\in[N]$ we get that
\begin{eqnarray}
\label{est_diff_byiyi_1}
    &&\mathbb{E}\Bigl[\max\limits_{0\leq i\leq k}|\bar y^{l+1}_i-y^{l+1}_i|^p\Bigr]\leq L_f^pC_p\tau^{p-1}h\sum\limits_{j=1}^k\mathbb{E}\Bigl[\max\limits_{0\leq i\leq j-1}|\bar y^{l+1}_i-y^{l+1}_i|^p\Bigr]\notag\\
    &&+C_pL_f^p\tau^{p-1}h\sum\limits_{j=1}^k\mathbb{E}\Bigl[|\phi_l(\delta_j^{l+1})-y_{j-1}^l|^{\alpha_1 p}\Bigr]+C_p\mathbb{E}\Bigl[\max\limits_{1\leq i\leq k}\Bigl|\sum\limits_{j=1}^iG_{j-1}^{l+1}\cdot \Delta_{j-1}^{l+1}W\Bigl|^p\Bigr]
\end{eqnarray}
Moreover
\begin{align*}
\mathbb{E}[|\phi_{l}(\delta_{j}^{l+1})-y^{l}_{j-1}|^{\alpha_1 p}] 
&\leq C_p \mathbb{E}[|\phi_{l}(\delta_{j}^{l+1})-\phi_{l}(t_{j-1}^0)|^{\alpha_1 p}]+C_p \mathbb{E}\Bigl[\max\limits_{0\leq k\leq N}|\phi_{l}(t_k^0)-y_k^l|^{\alpha_1 p}\Bigr]
\end{align*}
where by Theorem \ref{sol_dde_prop_1}, Jensen inequality (applied to the concave function $[0,+\infty)\ni x\to x^{\alpha_1}$), and \eqref{ind_assumpt_1} we have 
\begin{equation}
    \label{est_01}
    \mathbb{E}[|\phi_{l}(\delta_{j}^{l+1})-\phi_{l}(t_{j-1}^0)|^{\alpha_1 p}]\leq \Bigl(\mathbb{E}[|\phi_{l}(\delta_{j}^{l+1})-\phi_{l}(t_{j-1}^0)|^{ p}]\Bigr)^{\alpha_1}\leq \bar K^{\alpha_1}_l h^{\alpha_1p/2},
\end{equation}
and
\begin{equation}
\label{est_1}
    \mathbb{E}\Bigl[\max\limits_{0\leq k\leq N}|\phi_{l}(t_k^0)-y_k^l|^{\alpha_1 p}\Bigr]\leq\Biggl( \mathbb{E}\Bigl[\max\limits_{0\leq k\leq N}|\phi_{l}(t_k^0)-y_k^l|^{p}\Bigr]\Biggr)^{\alpha_1}.
\end{equation}
Now define for all $t\in [0,t_N^{l+1}]$ the following stochastic process
\begin{align*}
   M(t):=  \int_0^t\sum\limits_{j=1}^N
   G_{j-1}^{l+1}\cdot
   {\bf 1}_{[t_{j-1}^{l+1},t_{j}^{l+1})}(s)\,\mathrm{d}W(s).
\end{align*}
Note that by \eqref{def_Glj}, \eqref{sig_filed2}, and \eqref{meas_bar_ykj} we have that
\begin{equation}
   \sigma(G_{j-1}^{l+1})\subset\tilde\Sigma_{t_{j-1}^{l+1}}=\Sigma^{l+1}_{t_{j-1}^0}\vee\sigma(\{\gamma_k^j\}_{j\in\mathbb{N}_0,k\in\mathbb{N}}), 
\end{equation} 
and, hence, $G_{j-1}^{l+1}$ is independent of $\Delta_{j-1}^{l+1}W$. Therefore, the stochastic It\^o integral above is well-defined. Moreover, the quadratic variation of the martingale $M$  for $t\in [0,t_N^{l+1}]$ is as follows
\begin{align*}
    \langle M\rangle(t)&=\int_0^t\sum\limits_{j=1}^N |G_{j-1}^{l+1}|^2\cdot {\bf 1}_{[t_{j-1}^{l+1},t_{j}^{l+1})}(s)\,\mathrm{d}s
\end{align*}
and then for $k\in [N]$
\begin{align*}
    \langle M\rangle (t_k^{l+1}
) &= h\sum\limits_{j=1}^k |G_{j-1}^{l+1}|^2.
\end{align*}
 From Burkholder-Davis-Gundy inequality (see, for example, Corollary 2.9, page 82 in \cite{PARRAS}), Jensen's inequality and the assumption  \eqref{assumpt_b32} on $g$, we have that 
 \begin{eqnarray}
 \label{est_int_Gij}
     &&\mathbb{E}\Bigl[\max\limits_{1\leq i\leq k}\Bigl|\sum\limits_{j=1}^iG_{j-1}^{l+1}\cdot \Delta_{j-1}^{l+1}W\Bigl|^p\Bigr]=\mathbb{E}\Bigl[\max\limits_{1\leq i \leq k}|M(t_i^{l+1})|^p\Bigr]\leq \mathbb{E}\Bigl[\sup\limits_{0\leq t\leq t_k^{l+1}}|M(t)|^p\Bigr]\notag\\
     &&\leq C_p\mathbb{E}\Bigl( \langle M\rangle (t_k^{l+1}
)\Bigr)^{p/2}\leq C_p\tau^{\frac{p}{2}-1}h\sum\limits_{j=1}^k|G_{j-1}^{l+1}|^p\notag\\
&&\leq \bar C_p\tau^{\frac{p}{2}-1}L_g^ph\sum\limits_{j=1}^k\mathbb{E}\Bigl[\max\limits_{0\leq i\leq j-1}|\bar y^{l+1}_i-y^{l+1}_i|^p\Bigr]\notag\\
&&\quad\quad+\bar C_p\tau^{\frac{p}{2}}L_g^p\Bigl(\mathbb{E}\Bigl[\max\limits_{0\leq k\leq N}|\phi_l(t_k^0)-y_k^l|^{p}\Bigr]\Bigr)^{\alpha_2}.
 \end{eqnarray}
 Combining \eqref{est_diff_byiyi_1}, \eqref{est_01}, \eqref{est_1}, \eqref{est_int_Gij} and using the fact that $y_0^{l+1}=\bar y_0^{l+1}$ we arrive  for $k\in [N]$ at
 \begin{eqnarray}
     &&\mathbb{E}\Bigl[\max\limits_{0\leq i\leq k}|\bar y^{l+1}_i-y^{l+1}_i|^p\Bigr]\leq c_1 h\sum\limits_{j=1}^{k-1}\mathbb{E}\Bigl[\max\limits_{0\leq i\leq j}|\bar y^{l+1}_i-y^{l+1}_i|^p\Bigr]+c_2 h^{\alpha_1 p/2}\notag\\
     &&+c_3\Biggl( \mathbb{E}\Bigl[\max\limits_{0\leq k\leq N}|\phi_{l}(t_k^0)-y_k^l|^{p}\Bigr]\Biggr)^{\alpha_1}+c_4\Biggl( \mathbb{E}\Bigl[\max\limits_{0\leq k\leq N}|\phi_{l}(t_k^0)-y_k^l|^{p}\Bigr]\Biggr)^{\alpha_2}.
 \end{eqnarray}
 By the discrete version of Gronwall's lemma (see, for example, Lemma 2.1. in 
 \cite{RKYW2017}) we get
 \begin{eqnarray}
 \label{est_diff_bykyk}
     &&\mathbb{E}\Bigl[\max\limits_{0\leq k\leq N}|\bar y^{l+1}_k-y^{l+1}_k|^p\Bigr]\leq K_{1,l}e^{K_{2,l}\tau}\Biggl[h^{\alpha_1 p/2}+\Biggl( \mathbb{E}\Bigl[\max\limits_{0\leq k\leq N}|\phi_{l}(t_k^0)-y_k^l|^{p}\Bigr]\Biggr)^{\alpha_1}\notag\\
     &&\quad\quad+\Biggl( \mathbb{E}\Bigl[\max\limits_{0\leq k\leq N}|\phi_{l}(t_k^0)-y_k^l|^{p}\Bigr]\Biggr)^{\alpha_2}\Biggr].
 \end{eqnarray}
We now establish an upper bound on  $\displaystyle{\Bigl\|\max\limits_{0\leq i\leq N}|\phi_{l+1}(t_i^0)-\bar y_i^{l+1}|\Bigl\|_{L^p(\Omega)}}$. For $k\in [N]$ we have
\begin{eqnarray}
        &&\phi_{l+1}(t_k^0)-\bar y^{l+1}_k=\phi_{l+1}(0)-\bar y_0^{l+1}+(\phi_{l+1}(t_k^0)-\phi_{l+1}(t_0^0))-(\bar y^{l+1}_k-\bar y^{l+1}_0)\notag\\
        &&=(\phi_l(t_N^0)-y_N^l)+\sum\limits_{j=1}^k(\phi_{l+1}(t_j^0)-\phi_{l+1}(t_{j-1}^0))-\sum\limits_{j=1}^k(\bar y^{l+1}_j-\bar y^{l+1}_{j-1})\notag\\
        &&=(\phi_l(t_N^0)-y_N^l)+\sum\limits_{j=1}^k\Biggl(\int\limits_{t_{j-1}^0}^{t_{j}^0}f_{l+1}(s,\phi_{l+1}(s))\,\de s-h f_{l+1}(\delta_j^{l+1},\bar y^{l+1}_{j-1})\Biggr)\notag\\
        &&\quad+\sum\limits_{j=1}^k\Biggl(\int\limits_{t_{j-1}^0}^{t_{j}^0}g_{l+1}(s,\phi_{l+1}(s))\,\de W_{l+1}(s)-g_{l+1}(t_{j-1}^{0}, \bar y_{j-1}^{l+1})\Delta_{j-1}^{l+1} W\Biggr)\notag\\
        \label{err_phi_byk_decomp1}
        &&=(\phi_l(t_N^0)-y_N^l)+\sum\limits_{i=1}^6 S^k_{i,l+1},
\end{eqnarray}
where
\begin{align*}
S^k_{1,l+1} &= \sum\limits_{j=1}^k \Biggl(\int\limits_{t_{j-1}^0}^{t_{j}^0}f_{l+1}(s,\phi_{l+1}(s))\,\de s-h f_{l+1}(\delta_j^{l+1},\phi_{l+1}(\delta_j^{l+1}))\Biggr), \\
S^k_{2,l+1} &= h\sum\limits_{j=1}^k\Bigl(f_{l+1}(\delta_j^{l+1},\phi_{l+1}(\delta_j^{l+1}))-f_{l+1}(\delta_j^{l+1},\phi_{l+1}(t_{j-1}^0))\Bigr), \\
S^k_{3,l+1} &= h\sum\limits_{j=1}^k\Bigl(f_{l+1}(\delta_j^{l+1},\phi_{l+1}(t_{j-1}^0))-f_{l+1}(\delta_j^{l+1},\bar y^{l+1}_{j-1})\Bigr),\\
S^k_{4,l+1} &= \int\limits_{0}^{t_{k}^0}\sum\limits_{j=1}^N\big(g_{l+1}(s,\phi_{l+1}(s))-g_{l+1}(t_{j-1}^{0}, \phi_{l+1}(s))\big)\mathbf{1}_{[t_{j-1}^0, t_{j}^0)}(s)\,\de W_{l+1}(s),\\
S^k_{5,l+1} &= \int\limits_{0}^{t_{k}^0}\sum\limits_{j=1}^N \big(g_{l+1}(t_{j-1}^{0},\phi_{l+1}(s))-g_{l+1}(t_{j-1}^{0}, \phi_{l+1}(t_{j-1}^{0}))\big)\mathbf{1}_{[t_{j-1}^0, t_{j}^0)}(s)\,\de W_{l+1}(s),\\
S^k_{6,l+1} &= \int\limits_{0}^{t_{k}^0}\sum\limits_{j=1}^N\big(g_{l+1}(t_{j-1}^{0}, \phi_{l+1}(t_{j-1}^{0}))-g_{l+1}(t_{j-1}^{0}, \bar y_{j-1}^{l+1})\big)\mathbf{1}_{[t_{j-1}^0, t_{j}^0)}(s)\,\de W_{l+1}(s).
\end{align*}
We have that
\begin{equation}
    S_{1,l+1}^k=\int\limits_0^{t_k^0}Y_{l+1}(s)\,\de s-h\sum\limits_{k=1}^k Y_{l+1}(\delta_j^{l+1}),
\end{equation}
where $\delta_j^{l+1}=t^0_{j-1}+h\gamma_{j}^{l+1}$,  $Y_{l+1}(t)=f_{l+1}(t,\phi_{l+1}(t))=f\big(t+(l+1)\tau, \phi_{l+1}(t),\phi_{l}(t)\big)$, and $\sigma((Y(t))_{t\in [0,\tau]})\subset\Sigma_{\infty}$. Hence, the process $Y$ is independent of the $\sigma$-field $\sigma(\{\gamma_k^j\}_{j\in\mathbb{N}_0,k\in\mathbb{N}})$. Moreover,  by Theorem \ref{sol_dde_prop_1} we arrive at
\begin{equation}
    \|f_{l+1}(\cdot,\phi_{l+1}(\cdot))\|_{L^p([0,\tau]\times\Omega;\mathbb{R}^d)}\leq T^{1/p}C_pK_f(1+K_l+K_{l+1})^{1/p}<+\infty.
\end{equation}
Therefore, by Theorem  4.1 in \cite{RKYW2019} we obtain 
\begin{equation}
\label{est_S1}
    \Bigl\|\max\limits_{1\leq k\leq N}|S^k_{1,l+1}|\Bigl\|_{L^p(\Omega)}\leq 2C_p\tau^{\frac{p-2}{2p}} (1+K_l)(1+K_{l+1}) h^{1/2}.
\end{equation}
By Jensen's inequality we get 
\begin{align*}
        \max\limits_{1\leq k\leq N}|S^k_{2,l+1}|^p\leq \tau^{p-1}L_f^ph\sum\limits_{j=1}^N |\phi_{l+1}(\delta_j^{l+1})-\phi_{l+1}(t_{j-1}^0)|^p .
\end{align*}
Thus, by Theorem \ref{sol_dde_prop_1}
\begin{equation}
\label{est_S2}
\begin{aligned}
    \Bigl\|\max\limits_{1\leq k\leq N}|S^k_{2,l+1}|\Bigl\|_{L^p(\Omega)}&\leq  L_f c_p\tau K(\tau^{1/2}+1)(1+K_{l}^{1/p}+K_{l+1}^{1/p})h^{1/2}.
\end{aligned}
\end{equation}
Moreover,
\begin{equation}
\label{est_S3}
\begin{aligned}
   &\mathbb{E}\Bigl[\max\limits_{1\leq i\leq k}|S^i_{3,l+1}|^p\Bigr]\leq C_pL_f^p \tau^{p-1} h\mathbb{E}|\phi_l(t_N^0)-y_N^l|^p\\
   &+C_p L_f^p\tau^{p-1} h\sum\limits_{j=1}^{k-1}\mathbb{E}\Bigl[\max\limits_{0\leq i \leq j}|\phi_{l+1}(t_i^0)-\bar y^{l+1}_i|^p\Bigr].
\end{aligned}
\end{equation}
Regarding $S_{4,l+1}^k$ we shall apply  Burkholder-Davis-Gundy inequality (see Corollary 2.9, page 82 in \cite{PARRAS}), Theorem \ref{sol_dde_prop_1} and  \eqref{eqn:gboundl+2} on $g_{l+1}$, and get
\vspace{-1cm}
\begin{align}\label{est_S4}
\begin{split}
      &\mathbb{E}\Bigl[\max\limits_{1\leq k\leq N}|S^k_{4,l+1}|^p\Bigr]\\
      &\leq C_p \mathbb{E}\Bigg( \int\limits_{0}^{t_{N}^0}\sum\limits_{j=1}^N\big|g_{l+1}(s,\phi_{l+1}(s))-g_{l+1}(t_{j-1}^{0}, \phi_{l+1}(s))\big|^2\mathbf{1}_{[t_{j-1}^0, t_{j}^0)}(s)\,\de s \Bigg)^{p/2}\\
      &=C_p \mathbb{E}\Bigg( \sum\limits_{j=1}^N\int\limits_{t_{j-1}^0}^{t_{j}^0}\big|g_{l+1}(s,\phi_{l+1}(s))-g_{l+1}(t_{j-1}^{0}, \phi_{l+1}(s))\big|^2\,\de s \Bigg)^{p/2}\\
      &\leq C_p N^{\frac{p}{2}-1} \sum\limits_{j=1}^N\mathbb{E}\Bigg(\int\limits_{t_{j-1}^0}^{t_{j}^0}\big|g_{l+1}(s,\phi_{l+1}(s))-g_{l+1}(t_{j-1}^{0}, \phi_{l+1}(s))\big|^2\,\de s \Bigg)^{p/2}\\
      &\leq C_p\tau^{\frac{p}{2}-1}\sum\limits_{j=1}^N\int\limits_{t_{j-1}^0}^{t_{j}^0}\mathbb{E}|g_{l+1}(s,\phi_{l+1}(s))-g_{l+1}(t_{j-1}^{0}, \phi_{l+1}(s))|^p\, \de s\\
      &\leq \tilde C_{1,l}  \sum\limits_{j=1}^N \int\limits_{t_{j-1}^0}^{t_{j}^0} |s-t_{j-1}^0|^{p\varrho}\cdot \mathbb{E}\big(1+|\phi_{l+1}(s)|+|\phi_{l}(s)|)^p\,\de s \\
      &\qquad + \tilde C_{2,l}  \sum\limits_{j=1}^N \int\limits_{t_{j-1}^0}^{t_{j}^0} \mathbb{E}\big|\phi_{l}(s)-\phi_{l}(t_{j-1}^0)\big|^{p\alpha_2}\,\de s \\
      &\leq \tilde K_{1,l} h^{p\varrho}+\tilde C_{2,l}  \sum\limits_{j=1}^N \int\limits_{t_{j-1}^0}^{t_{j}^0}\Bigl(\mathbb{E}\big|\phi_{l}(s)-\phi_{l}(t_{j-1}^0)\big|^{p}\Bigr)^{\alpha_2}\,\de s\\
      &\leq  \tilde K_{1,l} h^{p\varrho}+\tilde K_{2,l} h^{p\alpha_2/2}\leq \tilde K_{3,l}h^{p\min\{\varrho,\alpha_2/2\}}.
  \end{split}
\end{align}
Following the similar arguments as in Eqn.\eqref{est_S4}, and utilising Theorem \ref{sol_dde_prop_1} together with the  Assumption \eqref{eqn:gboundl+2} on $g_{l+1}$ we get 
\begin{align}\label{est_S5}
\begin{split}
          &\mathbb{E}\Bigl[\max\limits_{1\leq k\leq N}|S^k_{5,l+1}|^p\Bigr]\\&\leq C_p \mathbb{E}\Bigg( \int\limits_{0}^{t_{N}^0}\sum\limits_{j=1}^N\big|g_{l+1}(t_{j-1}^{0},\phi_{l+1}(s))-g_{l+1}(t_{j-1}^{0}, \phi_{l+1}(t_{j-1}^{0}))\big|^2\mathbf{1}_{[t_{j-1}^0, t_{j}^0)}(s)\,\de s \Bigg)^{p/2}\\
      & \leq C_pL_g^p  \mathbb{E}\Bigg( \sum\limits_{j=1}^N\int\limits_{t_{j-1}^0}^{t_{j}^0}\big|\phi_{l+1}(s)- \phi_{l+1}(t_{j-1}^{0})\big|^2\,\de s \Bigg)^{p/2}\\
      &\leq C_p L_g^p N^{\frac{p}{2}-1} \sum\limits_{j=1}^N \mathbb{E}\Bigg(  \int\limits_{t_{j-1}^0}^{t_{j}^0}\big|\phi_{l+1}(s)- \phi_{l+1}(t_{j-1}^{0})\big|^2\,\de s \Bigg)^{p/2}\\
      &\leq C_pL_g^p \tau^{\frac{p}{2}-1} \sum\limits_{j=1}^N  \int\limits_{t_{j-1}^0}^{t_{j}^0}\mathbb{E}\big[\big|\phi_{l+1}(s)- \phi_{l+1}(t_{j-1}^{0})\big|^p\big]\,\de s \leq \tilde K_{4,l} h^{p/2}.
     % &\leq C_p \tau^{p/2}c_p^p L_g^{p}K^p(\tau^{1/2}+1)^p(1+K_{l}^{1/p}+K_{l+1}^{1/p})^p h^{p/2}.
\end{split}
\end{align}
For the last term $S^k_{6,l+1}$ we use similar arguments as above and obtain
\begin{align*}
     & \mathbb{E}\Bigl[\max\limits_{1\leq i\leq k}|S^i_{6,l+1}|^p\Bigr] \leq C_p L_g^{p/2} \tau^{p/2-1} h\mathbb{E}|\phi_l(t_N^0)-y^l_N|^p\\
     &+C_p L_g^{p/2} \tau^{p/2-1} h \sum\limits_{j=1}^{k-1}  \mathbb{E}\Bigl[\max\limits_{0\leq i \leq j} \big| \phi_{l+1}(t_{i}^{0})-\bar y^{l+1}_i\big|^p\Bigr].
\end{align*}
Hence, from \eqref{err_phi_byk_decomp1} and \eqref{est_S3} we have for $k\in\{1,2,\ldots,N\}$
\begin{align*}
    &\mathbb{E}\Bigl[\max\limits_{0\leq i\leq k}|\phi_{l+1}(t_i^0)-\bar y^{l+1}_i|^p\Bigr]\leq K_1\mathbb{E}|\phi_l(t_N^0)-y_N^l|^p \\
    &+ c_p\sum\limits_{m\in\{1,2,4,5\}}\mathbb{E}\Bigl[\max\limits_{1\leq k\leq N}|S^k_{m,l+1}|^p\Bigr] \\
    &+K_2 h \sum\limits_{j=1}^{k-1}  \mathbb{E}\Bigl[\max\limits_{0\leq i \leq j} \big| \phi_{l+1}(t_{i}^{0})-\bar y^{l+1}_i\big|^p\Bigr].
\end{align*}
Now by using  Gronwall's lemma (see,  Lemma 2.1 in \cite{RKYW2017}), \eqref{ind_assumpt_1}, and \eqref{est_S1} to \eqref{est_S5} we get for all $k\in [N]$
\begin{align*}
    &\mathbb{E}\Bigl[\max\limits_{0\leq i\leq k}|\phi_{l+1}(t_i^0)-\bar y^{l+1}_i|^p\Bigr]\leq K_3\Bigl(\mathbb{E}|\phi_l(t_N^0)-y_N^l|^p \\
    &+ \sum\limits_{m\in\{1,2,4,5\}}\mathbb{E}\Bigl[\max\limits_{1\leq k\leq N}|S^k_{m,l+1}|^p\Bigr]\Bigl),
\end{align*}
which gives
\begin{equation}
\label{error_phi_byk2}
    \Bigl\|\max\limits_{0\leq i\leq N}|\phi_{l+1}(t_i^0)-\bar y_i^{l+1}|\Bigl\|^p_{L^p(\Omega)}\leq K_{5,l}\Bigl(\mathbb{E}\Bigl[\max\limits_{0\leq k\leq N}|\phi_{l}(t_k^0)-y^{l}_k|^p\Bigr]+h^{p\min\{\varrho,\alpha_2/2\}}\Bigr).
    %C_{l+1} h^{\min\{\varrho,\frac{1}{2}{\color{cyan}\alpha_2}\}\alpha^{l+1}}.
\end{equation}
Combining \eqref{error_decomp}, \eqref{est_diff_bykyk}, and \eqref{error_phi_byk2} we obtain
\begin{eqnarray}
     &&\Bigl\|\max\limits_{0\leq i\leq N}|\phi_{l+1}(t_i^0)- y_i^{l+1}|\Bigl\|_{L^p(\Omega)}\leq K_{6,l}\Bigl ( h^{\min\{\varrho,\alpha_1/2,\alpha_2/2\}}+\Bigl\|\max\limits_{0\leq i\leq N}|\phi_{l}(t_i^0)- y_i^{l}|\Bigl\|_{L^p(\Omega)}\notag\\
     &&\quad\quad+\Bigl\|\max\limits_{0\leq i\leq N}|\phi_{l}(t_i^0)- y_i^{l}|\Bigl\|^{\alpha_1}_{L^p(\Omega)}+\Bigl\|\max\limits_{0\leq i\leq N}|\phi_{l}(t_i^0)- y_i^{l}|\Bigl\|^{\alpha_2}_{L^p(\Omega)}\Bigr),
\end{eqnarray}
and therefore
\begin{equation}
     \Bigl\|\max\limits_{0\leq i\leq N}|\phi_{l+1}(t_i^0)- y_i^{l+1}|\Bigl\|_{L^p(\Omega)}\leq C_{l+1} h^{\min\{\varrho,\frac{1}{2}\}\alpha^{l+1}},
\end{equation}
which ends the inductive part of the proof. Finally,  $\phi_{l+1}(t_i^0)=\phi_{l+1}(ih)=X(ih+(l+1)\tau)=X(t_i^{l+1})$ and the proof of \eqref{error_main_thm} is finished. 
\end{proof}
%%%%%%%%%%%%%%%%%%%%%%%%
%%%%%%%%%%%%%%%%%%%%%%%%
\begin{rem}
We briefly comment on optimality of the defined algorithm in the Information-Based Complexity sense, see \cite{TWW88}. In the special case $\alpha_1=\alpha_2=1$ we have the optimal bound $\Theta(h^{\min\{\varrho,1/2\}})$, which follows from Proposition 5.1 in \cite{PPPM2014}.
\end{rem}
%%%%%%%%%%%%%%%%%%%%%%
\section{Numerical experiments}
In order to illustrate our theoretical findings we perform several numerical experiments. We chose the following exemplary right-hand side functions:
\begin{equation}\label{eq:f1}
    f_i(t,x,z)=k_i(t) \Bigl( x +0.01|z|^{\alpha_1}+ \sin (10 x) \cdot \cos(100|z|^{\alpha_1} ) \Bigr),
\end{equation}
and the diffusion is either additive \begin{equation}\label{eq:g1}
    g_1(t,x,z)=0.5 |\cos(2^5\pi t)| ,
\end{equation}
or multiplicative 
\begin{equation}\label{eq:f2}
    g_2(t,x,z)=k(t) \Bigl( x +0.01|z|^{\alpha_2}+ \cos (10 x) \cdot \cos(100|z|^{\alpha_2} ) \Bigr),
\end{equation}
where, $\alpha_1,\alpha_2\in (0,1]$, $k_1$ is the following periodic function
\begin{equation}\label{eqn:k1}
    k_1(t)=\sum_{j=0}^n \Bigl((j+1)\tau-t\Bigr)^{-1/\gamma_1}\cdot\mathbf{1}_{[j\tau,(j+1)\tau]}(t),\quad \gamma_1> 2,\end{equation}
which belongs to $L^p([0,(n+1)\tau])$,  $k_2$ is a step function satisfying
\begin{align}\label{eqn:k2}
    k_2(t)=\sum_{j=0}^n 0.1\cdot(j+1)\cdot\mathbf{1}_{[j\tau,(j+1)\tau]}(t),  
\end{align} and $k$ is given by
\begin{equation}
    k(t)=t^{\gamma_2}, \quad \gamma_2\in (0,1],
\end{equation}
which is $\gamma_2$-H\"older continuous. Note that formally $f_1$ does not satisfy our assumptions, however, we get numerical evidence that we probably could extend result of Theorem \ref{rate_of_conv_expl_Eul} to the case when $L_f,L_g$ are integrable funtions.\\
We implement randomized Euler-Maruyama scheme \eqref{expl_euler_1}-\eqref{expl_euler_11} using Python programming language. Moreover, since for the right-hand side functions \eqref{eq:f1} plus \eqref{eq:g1} or  \eqref{eq:f1} plus \eqref{eq:f2} we do not know the exact solution $X(t)$, we  approximate the mean square error  \eqref{eq:L2err} for each $0\leq j\leq n$ with
\[
\left\|\max_{0\leq i\leq N}|\tilde y_i^j-y_i^j|\right\|_{L^2(\Omega)}\approx\left(\frac1K\sum_{k=1}^{K}\max_{0\leq i\leq N}|\tilde y_i^j(\omega_k)-y_i^j(\omega_k)|^2\right)^\frac12,
\]
where $K\in\mathbb{N}$, $\{\omega_k\}_{k=1}^K$ represents the $k$th realisation from the complete probability space, $y_i^j$ is the output of the randomized Euler-Maruyama scheme on the initial mesh $t_i^j\udef j\tau+ih$ and $h\udef\frac{\tau}{N}$ for $i=0,\ldots,N-1$, while $\tilde y_i^j$ is the  reference solution obtained also from the randomized Euler-Maruyama scheme but on the refined mesh $\tilde t_i^j\udef j\tau+i\tilde h$ and $\tilde h\udef\frac{h}{m}=\frac{\tau}{mN}$ for $i=0,\ldots,mN-1$. Note that $\{\omega_k\}_{k=1}^K$ is generated on the refined mesh.\\

The implementation of the randomized Euler-Maruyama method
method is straightforward. To evaluate the solution at time point $t^j_i=ih+j\tau$ within the interval $[j\tau,(j+1)\tau]$, we need two steps:

{\bf Step (j1):} First simulate $\gamma \sim \mathcal{U}(0,1)$ and set random time point $
    t^j_{i-1}+ \gamma h\in [t^j_{i-1},t^j_i]$.
    
{\bf Step (j2):}  Compute $y_{i}^j$ as defined in \eqref{expl_euler_1} and \eqref{expl_euler_11}.

Listing~\ref{py:RandEM} shows an implementation  of method \eqref{expl_euler_1} and \eqref{expl_euler_11} in
the case of a 1-dimensional Wiener process ($m=1$) in \textsc{Python}.

\begin{lstlisting}[caption={A sample implementation of \eqref{expl_euler_1} and \eqref{expl_euler_11} in
\textsc{Python}}, 
language=Python, label=py:RandEM]
import numpy as np

def f(t,x,z):
    return [...]

def g(t,x,z):
    return [...]


def randEM_full(tau, X0, h, f,g,n_taus):

    # input:    delay lag tau, stepsize h, initials X0, 
    #            drift and diffusion functions f and g
    #            number of intervals of length tau n_taus
    # output:   
    #  one trajectory of the randomized Euler-Maruyama method

    ###to get numerical evaluation:
    
    #number of steps within one interval with length tau
    N=int(tau/stepsize) 
                    
    #setting initial conditions 
    sol = np.zeros((N+ 1,n_tau+1))+X0
        
    #for each interval [j*tau, (j+1)*tau]     
    for j in range(1, n_taus+1):           
            
	sol[0,j] = sol[-1,j-1]
	
	# collect the grid points over 	[j*tau, (j+1)*tau]   	
	grid = j* tau + h *np.arange(0, N) 

          
        for i in range(1,N+1):

	    # step (j1):	
            gamma=np.random.rand() 
            rand_time=grid[i-1]+gamma*h

            # step (j2):
            drift=h*f(rand_time, sol[i-1,j], sol[i-1,j-1])
            diff=g(grid[i-1], sol[i-1,j], sol[i-1,j-1]) \
                  *np.sqrt(h)*np.random.normal()  	    
                
            sol[i,j] = sol[i-1,j]+drift+diff

        
    return sol
\end{lstlisting}
\begin{exm} [Additive noise]
\normalfont
In the following numerical tests we use \eqref{eq:f1} with \eqref{eqn:k1} and \eqref{eq:g1} (So in this case we formally have $\alpha_2=1$.). We fix the number of experiments $K=1000$ for each $N=2^l$, $l=5,\ldots,10$, and the reference solution is computed with stepsize $2^{-17}$; also, the horizon parameter is $n=3$. We get the following results for $\gamma_1=3$: \\
letting $\alpha_1=0.1$, the negative mean square error slopes are 
$0.65$, $0.65$, and $0.62$. See Figure \ref{fig:alpha0.1,gamma3,g1}; \\
letting $\alpha_1=0.5$, the negative mean square error slopes are
$0.65$, $0.64$, and $0.62$. See Figure \ref{fig:alpha0.5,gamma3,g1}; \\
letting $\alpha_1=1$,  the negative mean square error slopes are
$0.63$, $0.63$, and $0.58$. See Figure \ref{fig:alpha1,gamma3,g1}; \\
while, for $\gamma_1=5$: \\
letting $\alpha_1=0.1$, the negative mean square error slopes are
$0.60$, $0.63$, and $0.61$. See Figure \ref{fig:alpha0.1,gamma5,g1}; \\
letting $\alpha_1=0.5$, the negative mean square error slopes are
$0.65$, $0.65$, and $0.62$. See Figure \ref{fig:alpha0.5,gamma5,g1}; \\
letting $\alpha_1=1$, the negative mean square error slopes are
$0.63$,  $0.63$ and $0.59$. See Figure \ref{fig:alpha1,gamma5,g1}.

\begin{figure}
\centering
\begin{subfigure}{0.32\textwidth}
    \centering
    \includegraphics[width=\textwidth]{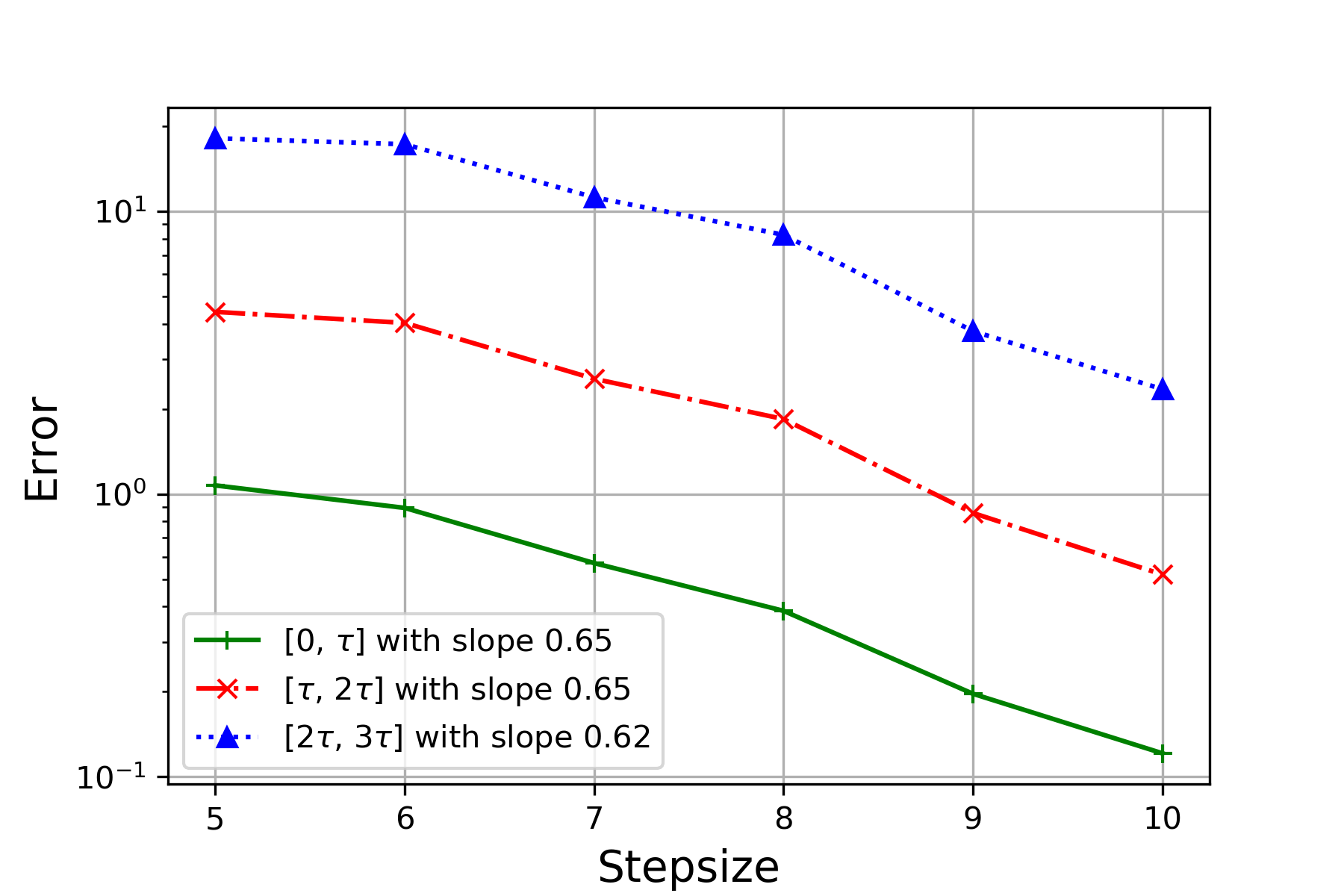}
    \caption{$\alpha=0.1$. \label{fig:alpha0.1,gamma3,g1}}
\end{subfigure}
\begin{subfigure}{0.32\textwidth}
    \centering
    \includegraphics[width=\textwidth]{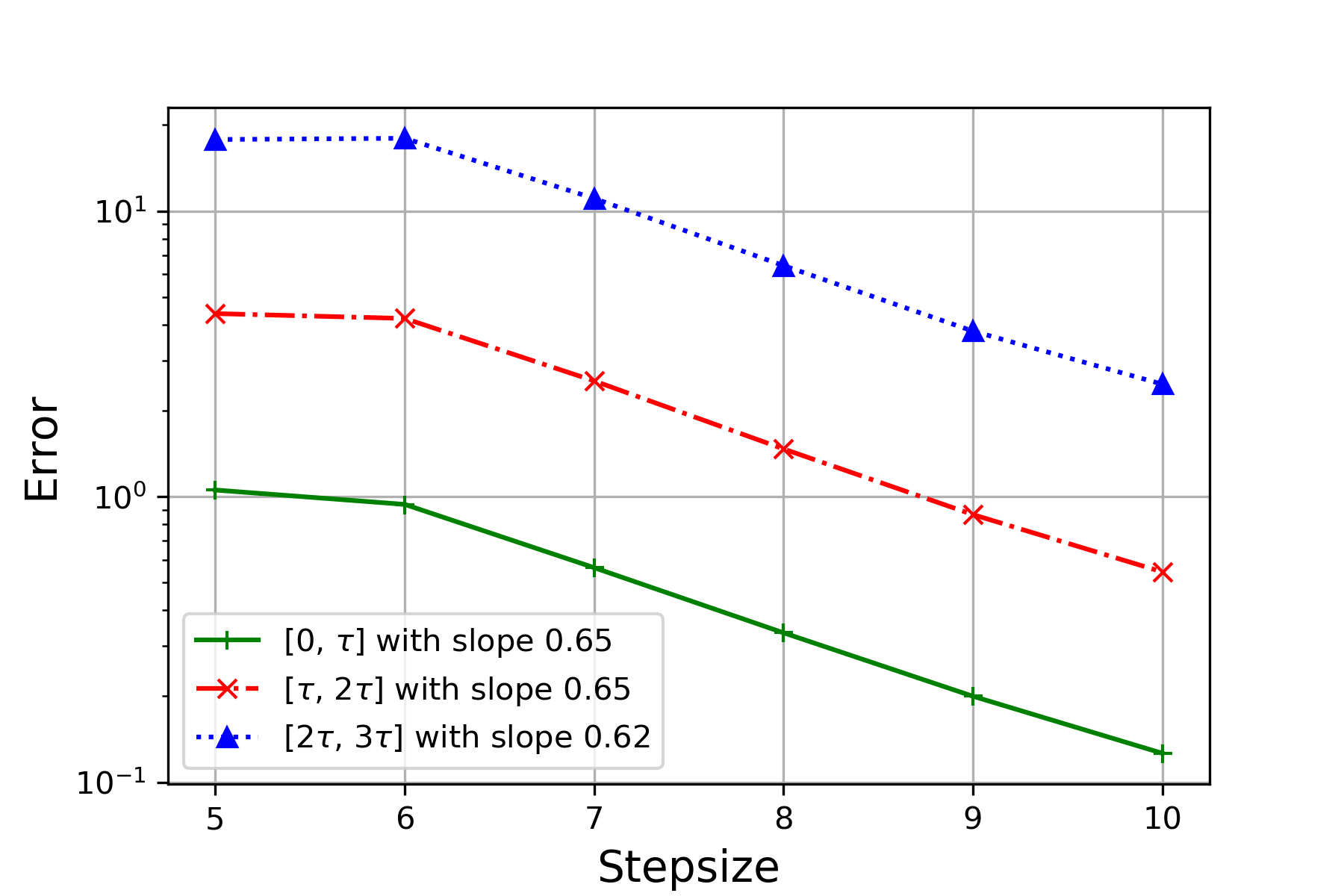}
    \caption{$\alpha=0.5$. \label{fig:alpha0.5,gamma3,g1}}
\end{subfigure}
\begin{subfigure}{0.32\textwidth}
    \centering
    \includegraphics[width=\textwidth]{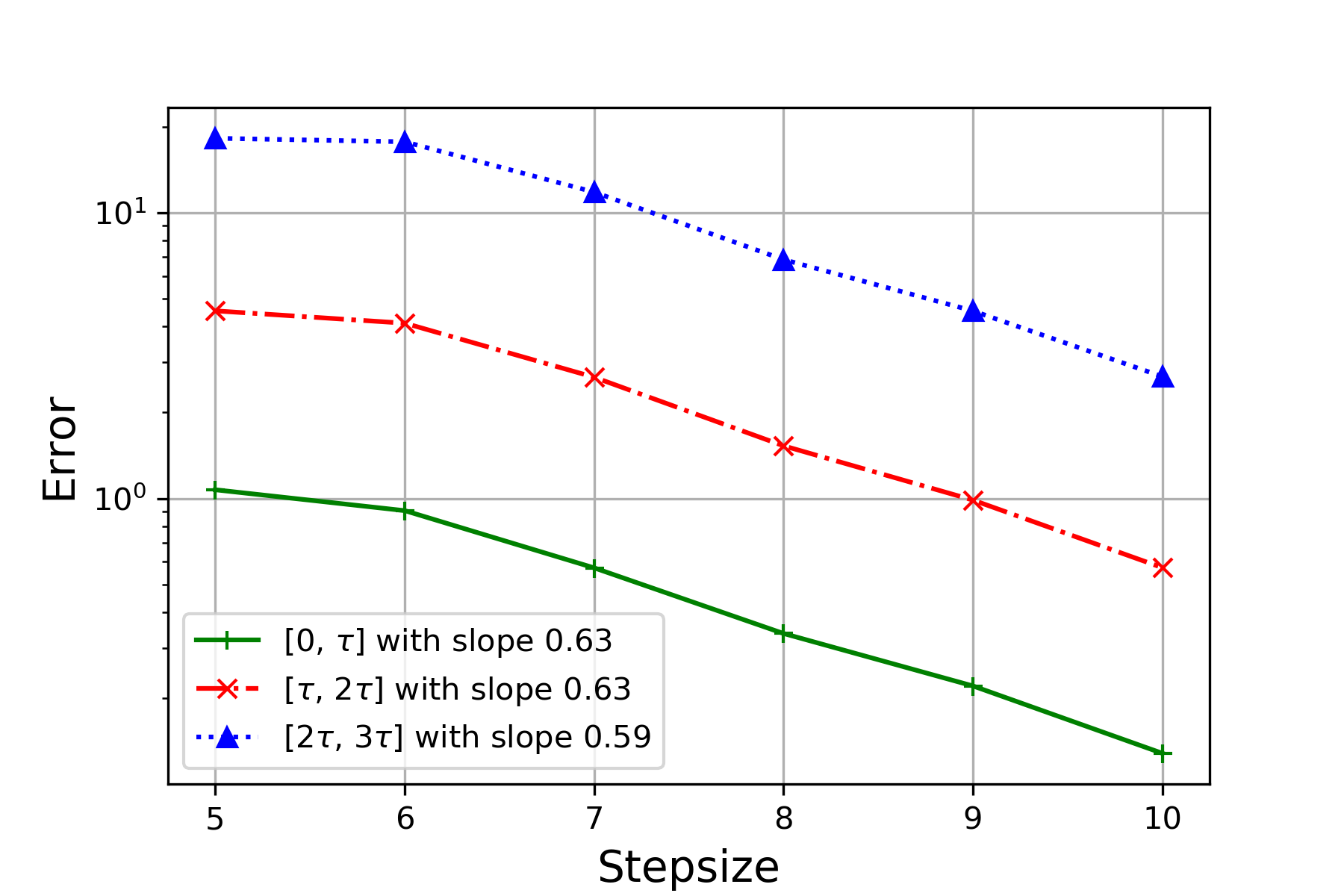}
    \caption{$\alpha=1$.\label{fig:alpha1,gamma3,g1}}  
\end{subfigure}
\caption{Mean square errors slope for $\gamma=3$ and values of $\alpha=0.1,0.5,1$ for \eqref{eq:f1} plus \eqref{eq:g1}. \label{fig:mse_g1_gamma3}}
\end{figure}

\begin{figure}
\centering
\begin{subfigure}{0.32\textwidth}
    \centering
    \includegraphics[width=\textwidth]{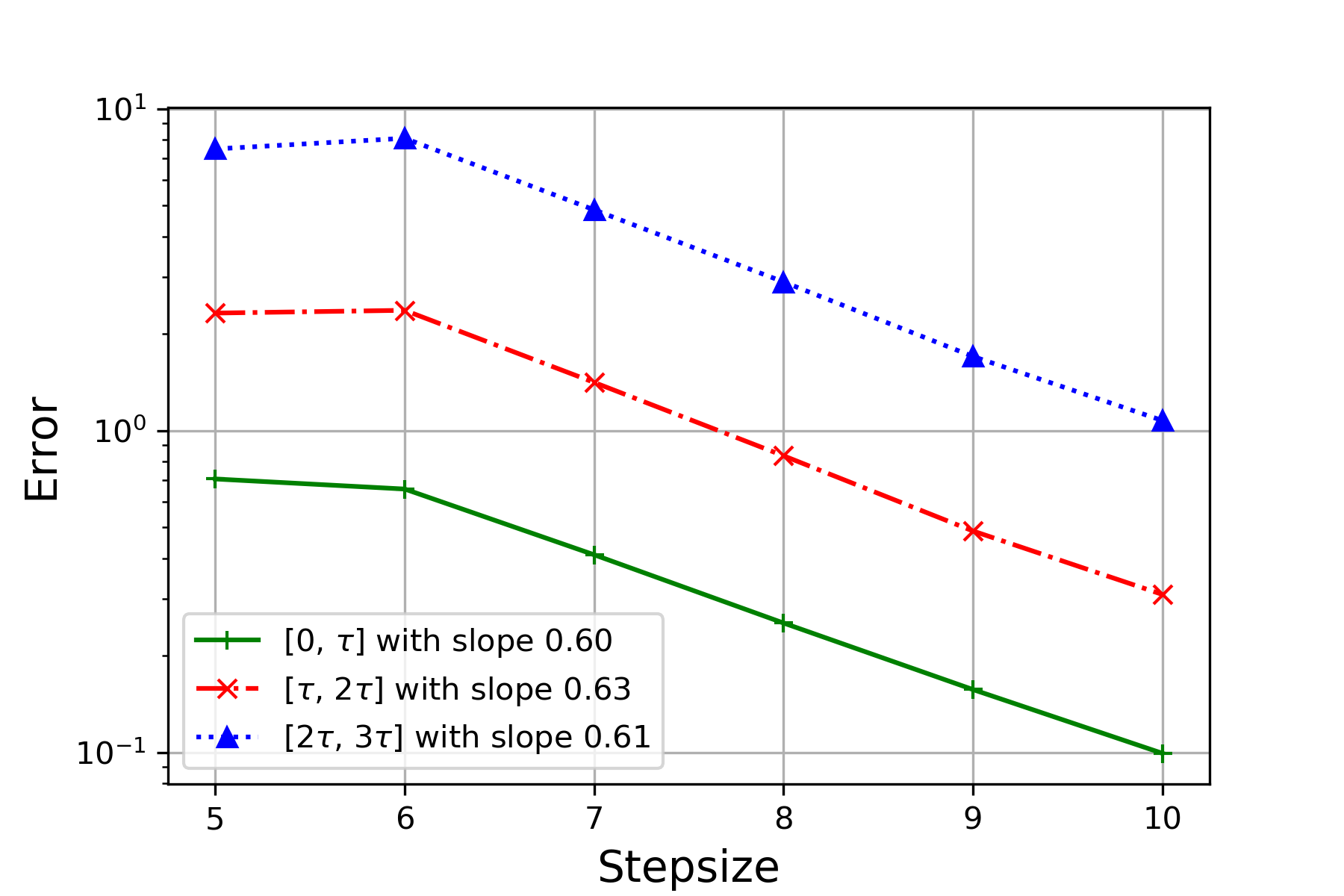}
    \caption{$\alpha_1=0.1$.}
    \label{fig:alpha0.1,gamma5,g1}
\end{subfigure}
\begin{subfigure}{0.32\textwidth}
    \centering
    \includegraphics[width=\textwidth]{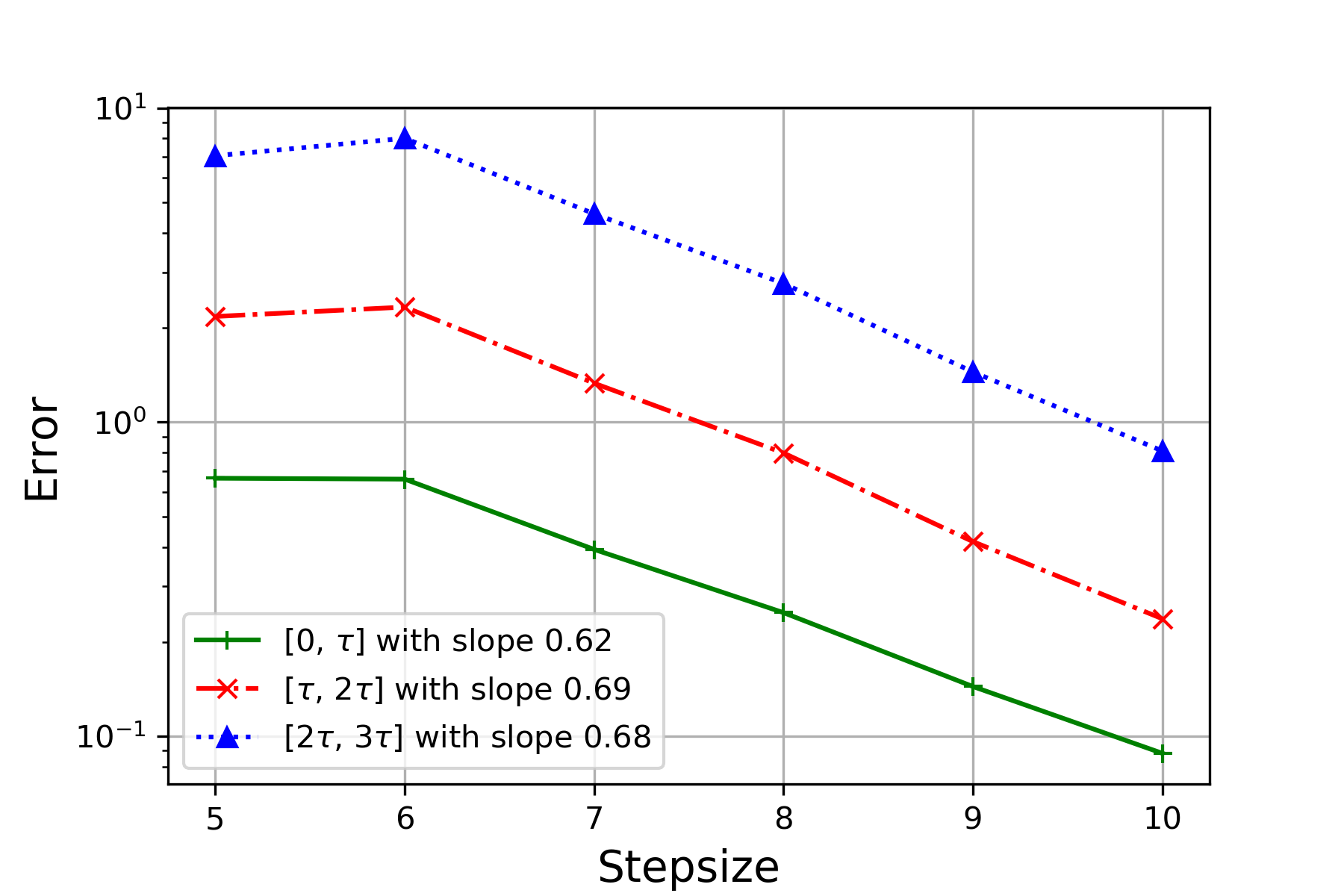}
    \caption{$\alpha_1=0.5$.}
    \label{fig:alpha0.5,gamma5,g1}
\end{subfigure}
\begin{subfigure}{0.32\textwidth}
    \centering
    \includegraphics[width=\textwidth]{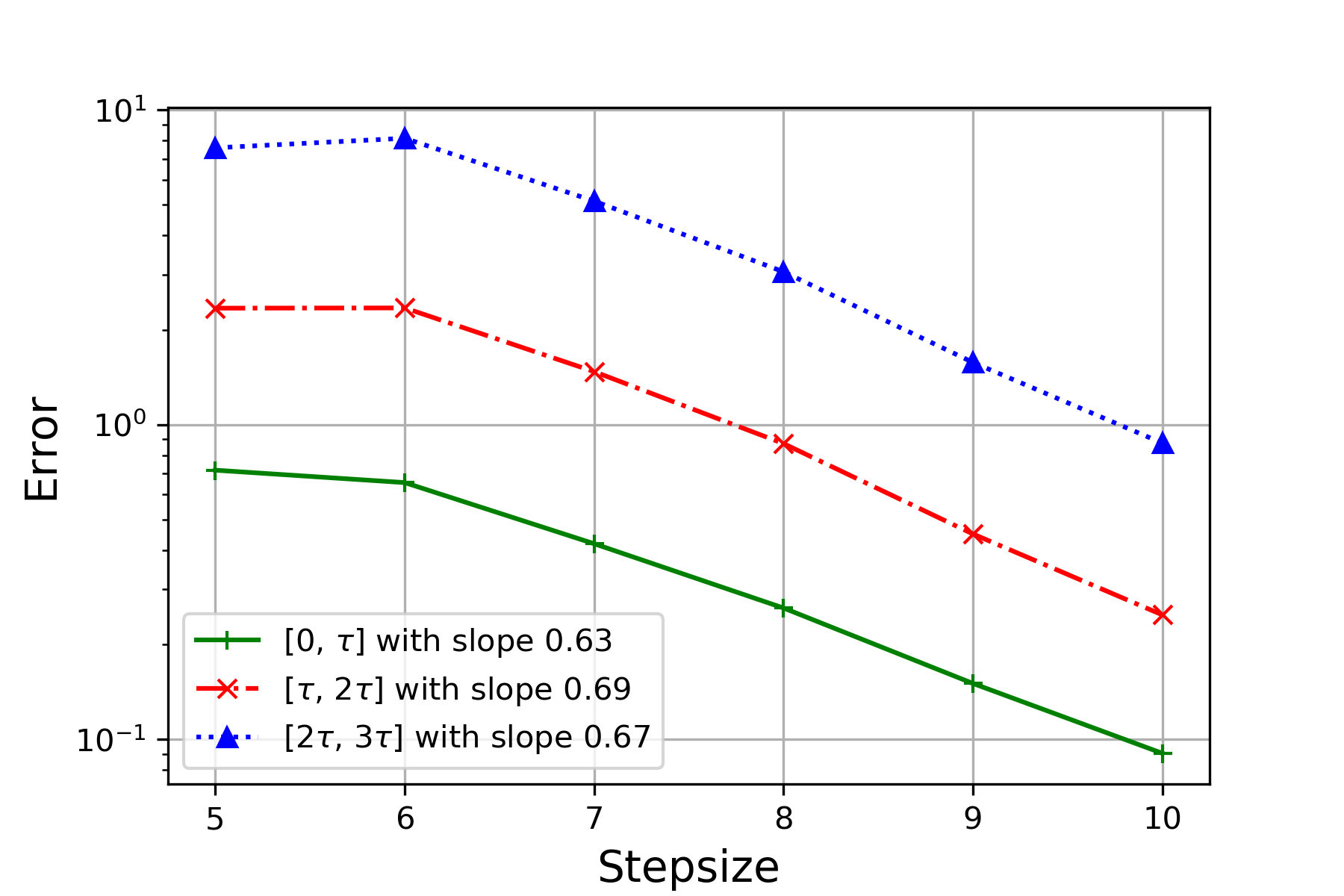}
    \caption{$\alpha_1=1$.}
    \label{fig:alpha1,gamma5,g1}
\end{subfigure}
\caption{Mean square errors slope for $\gamma_1=5$ and values of $\alpha_1=0.1,0.5,1$ for \eqref{eq:f1} plus \eqref{eq:g1}.}
\label{fig:mse_g1_gamma5}
\end{figure}
All the figures show that the error of $[j\tau,(j+1)\tau]$ is worse that the error of $[(j-1)\tau,j\tau]$ for $j\in \{1,2\}$ though the changes in the order of convergence from  $[(j-1)\tau,j\tau]$ to $[j\tau,(j+1)\tau]$ are not significant. 
\end{exm}
\begin{exm} [Multiplicative noise]
\normalfont
In the following numerical tests we use \eqref{eq:f1} with \eqref{eqn:k1} and \eqref{eq:f2}. We fix the number of experiments $K=1000$ for each $N=2^l$, $l=5,\ldots,10$, and the reference solution is computed with stepsize $2^{-17}$; also, the horizon parameter is $n=3$. We fix $\gamma_1=5$ and vary $\gamma_2$, $\alpha_1$ and $\alpha_2$. 

We get the following results for $\gamma_1=5$ and $\gamma_2=0.1$: \\
letting $\alpha_1=0.1$ and $\alpha_2=0.1$, the negative mean square error slopes are
$0.23$, $0.13$, and $0.11$. See Figure \ref{fig:alpha0.10.1,gamma51,f2}; \\
letting $\alpha_1=0.1$ and $\alpha_2=1$, the negative mean square error slopes are
$0.23$, $0.15$, and $0.14$. See Figure \ref{fig:alpha0.11,gamma51,f2}; \\
letting $\alpha_1=1$ and $\alpha_2=0.1$, the negative mean square error slopes are
$0.23$, $0.14$, and $0.11$. See Figure \ref{fig:alpha10.1,gamma51,f2}; \\
letting $\alpha_1=0.5$ and $\alpha_2=0.5$, the negative mean square error slopes are
$0.30$, $0.18$, and $0.22$. See Figure \ref{fig:alpha0.50.5,gamma51,f2}; \\
while, for $\gamma_1=5$ and $\gamma_2=0.5$: \\
letting $\alpha_1=0.1$ and $\alpha_2=0.1$, the negative mean square error slopes are $0.31$, $0.15$, and $0.17$. See Figure \ref{fig:alpha0.11,gamma50.5,f2}; \\
letting $\alpha_1=1$ and $\alpha_2=0.1$, the negative mean square error slopes are $0.31$, $0.15$, and $0.17$. See Figure \ref{fig:alpha10.1,gamma50.5,f2}; \\
letting $\alpha_1=0.1$ and $\alpha_2=1$, the negative mean square error slopes are $0.31$, $0.16$, and $0.23$. See Figure \ref{fig:alpha0.11,gamma50.5,f2}; \\
letting $\alpha_1=0.5$ and $\alpha_2=0.5$, the negative mean square error slopes are $0.31$, $0.18$, and $0.25$. See Figure \ref{fig:alpha0.50.5,gamma50.5,f2}; \\
while, for $\gamma_1=5$ and $\gamma_2=1$: \\
letting $\alpha_1=0.1$ and $\alpha_2=0.1$, the negative mean square error slopes are $0.31$, $0.28$, and $0.25$. See Figure \ref{fig:alpha0.11,gamma50.1,f2}; \\
letting $\alpha_1=1$ and $\alpha_2=0.1$, the negative mean square error slopes are $0.31$, $0.27$, and $0.24$. See Figure \ref{fig:alpha10.1,gamma50.1,f2}; \\
letting $\alpha_1=0.1$ and $\alpha_2=1$, the negative mean square error slopes are $0.31$, $0.25$, and $0.22$. See Figure \ref{fig:alpha0.11,gamma50.1,f2}; \\
letting $\alpha_1=0.5$ and $\alpha_2=0.5$, the negative mean square error slopes are $0.31$, $0.18$, and $0.25$. See Figure \ref{fig:alpha0.50.5,gamma50.1,f2}.
\begin{figure}
\centering
\begin{subfigure}{0.40\textwidth}
    \centering
    \includegraphics[width=\textwidth]{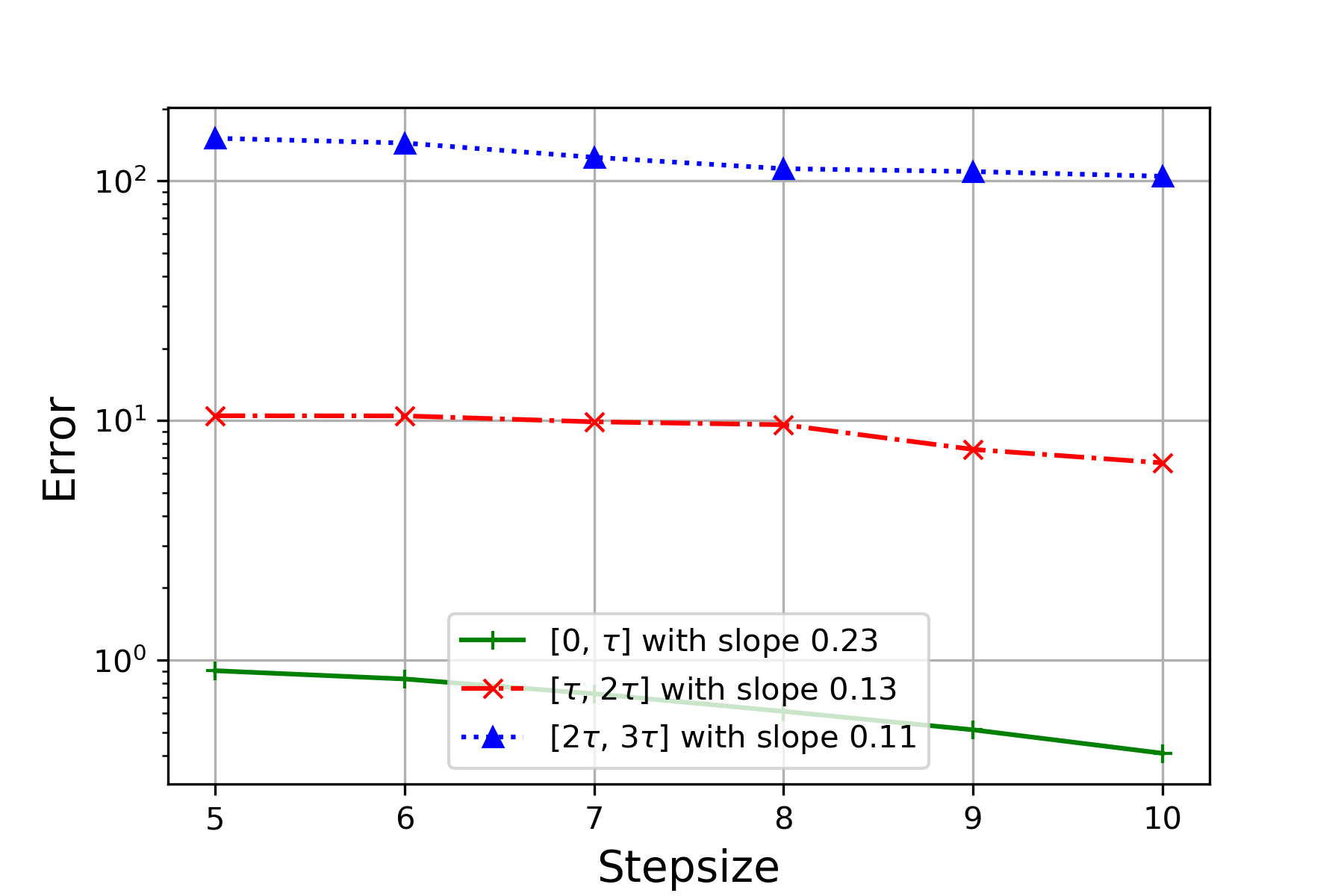}
    \caption{$\alpha_1=0.1$ and $\alpha_2=0.1$. \label{fig:alpha0.10.1,gamma51,f2}}
\end{subfigure}
\begin{subfigure}{0.40\textwidth}
    \centering
    \includegraphics[width=\textwidth]{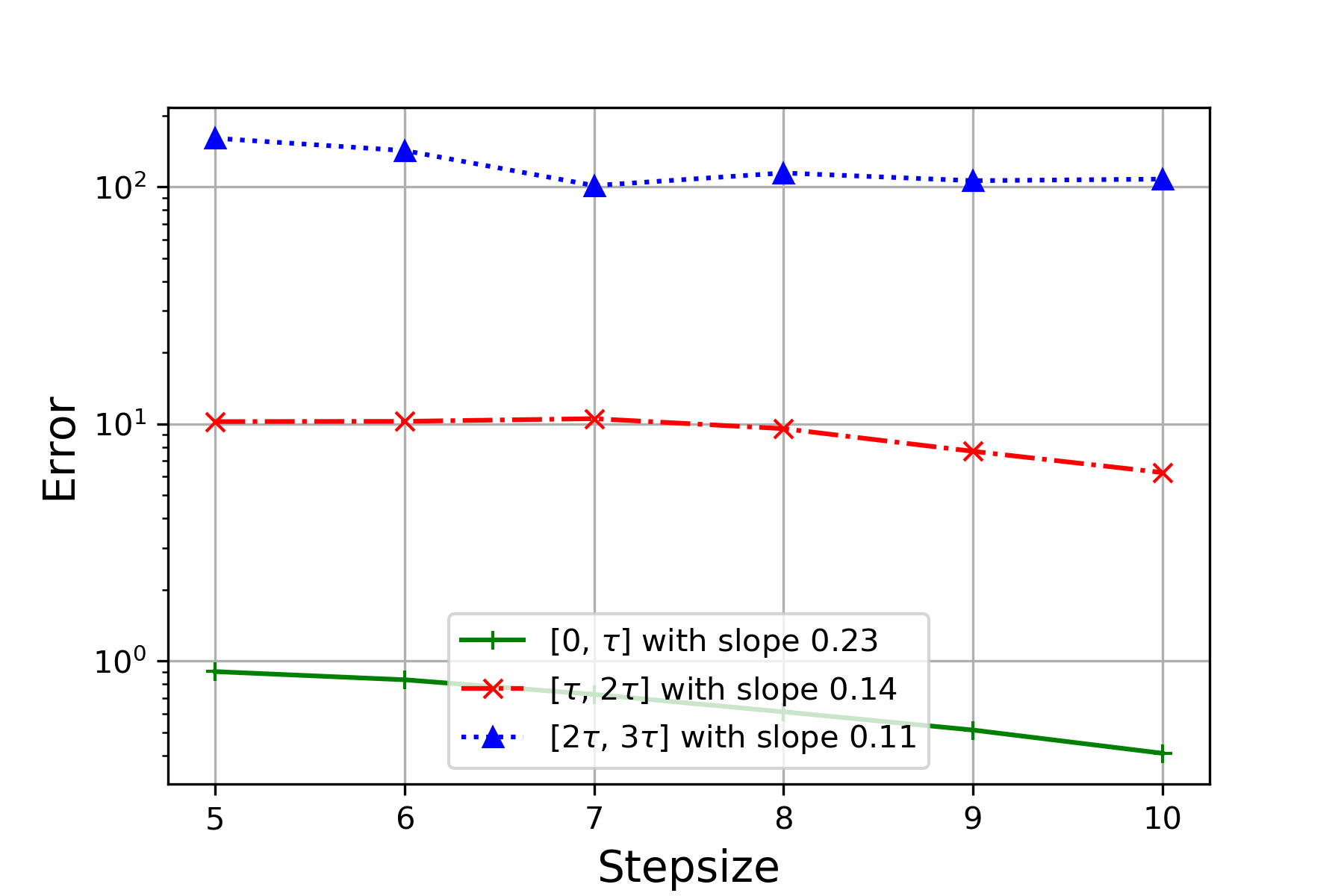}
    \caption{$\alpha_1=1$ and $\alpha_2=0.1$. \label{fig:alpha10.1,gamma51,f2}}
\end{subfigure}\\
\begin{subfigure}{0.4\textwidth}
    \centering
    \includegraphics[width=\textwidth]{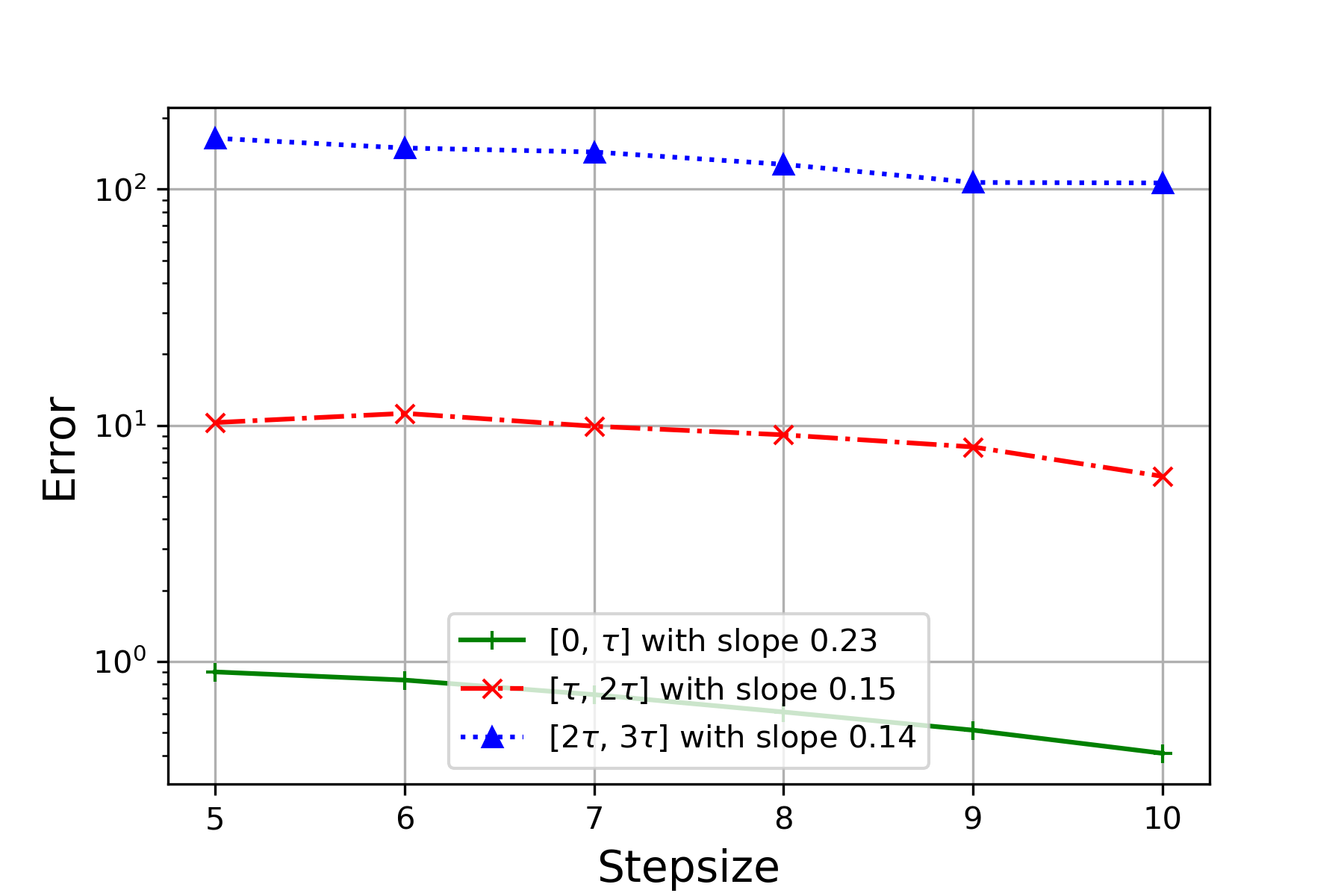}
    \caption{$\alpha_1=0.1$ and $\alpha_2=1$. \label{fig:alpha0.11,gamma51,f2}}  
\end{subfigure}
\begin{subfigure}{0.4\textwidth}
    \centering
    \includegraphics[width=\textwidth]{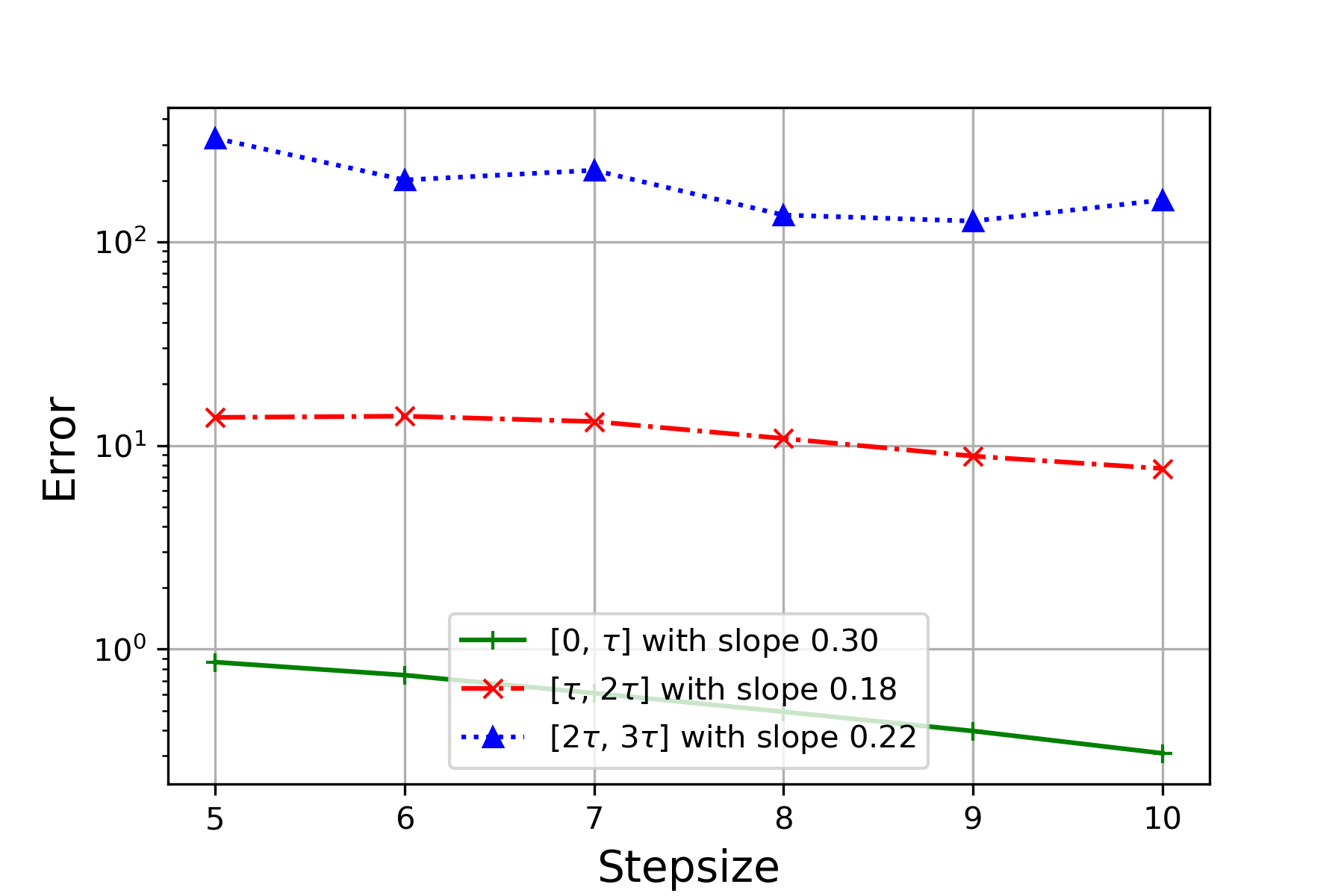}
    \caption{$\alpha_1=0.5$ and $\alpha_2=0.5$. \label{fig:alpha0.50.5,gamma51,f2}}  
\end{subfigure}
\caption{Mean square errors slope for $\gamma_1=5$ and $\gamma_2=0.1$ and values of $(\alpha_1,\alpha_2)=(0.1,0.1),(1,0.1),(0.1,1),(0.5,0.5)$ for \eqref{eq:f1} plus \eqref{eq:f2}. \label{fig:mse_g1f1_gamma51}}
\end{figure}

\begin{figure}
\centering
\begin{subfigure}{0.40\textwidth}
    \centering
    \includegraphics[width=\textwidth]{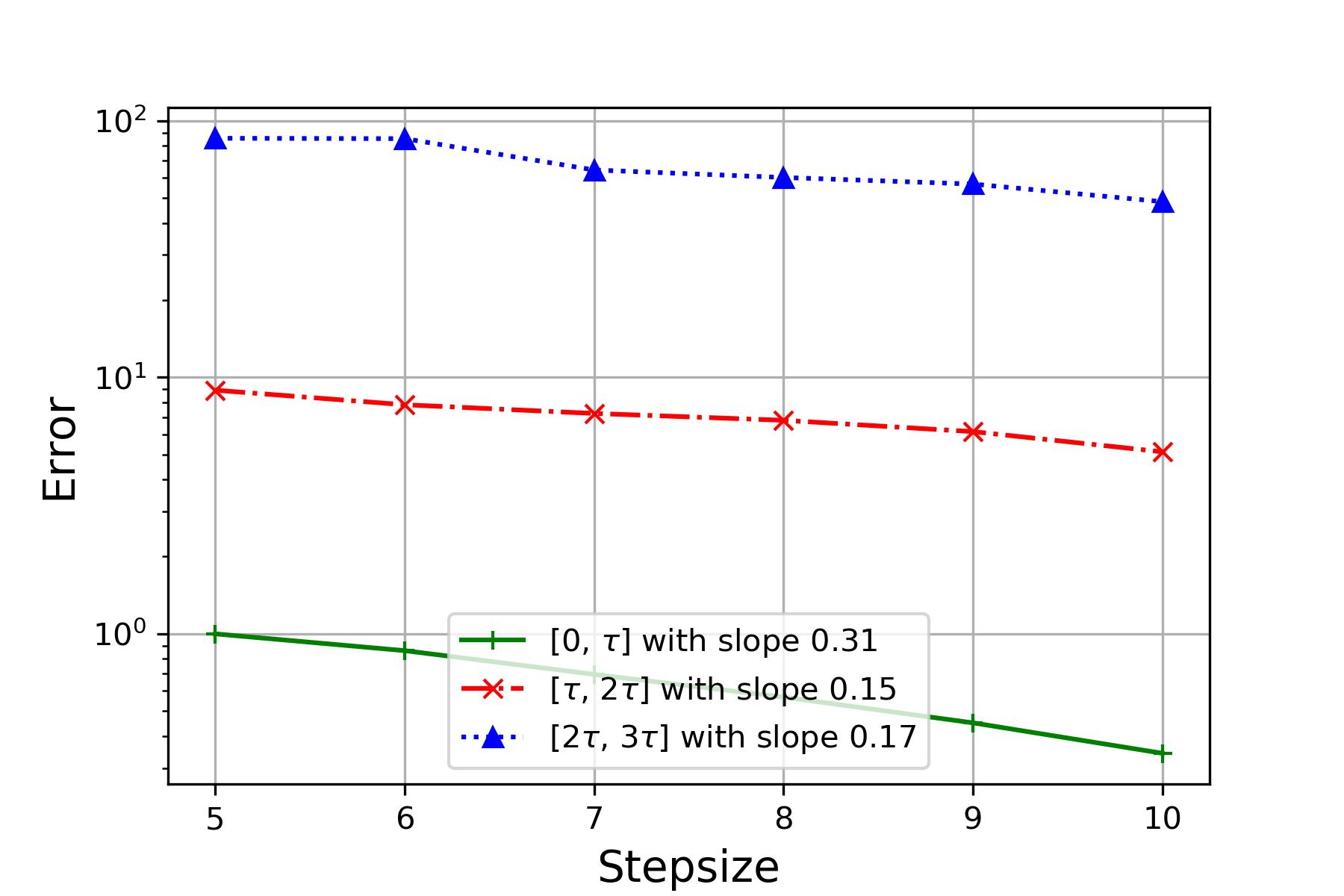}
    \caption{$\alpha_1=0.1$ and $\alpha_2=0.1$. \label{fig:alpha0.10.1,gamma50.5,f2}}
\end{subfigure}
\begin{subfigure}{0.40\textwidth}
    \centering
    \includegraphics[width=\textwidth]{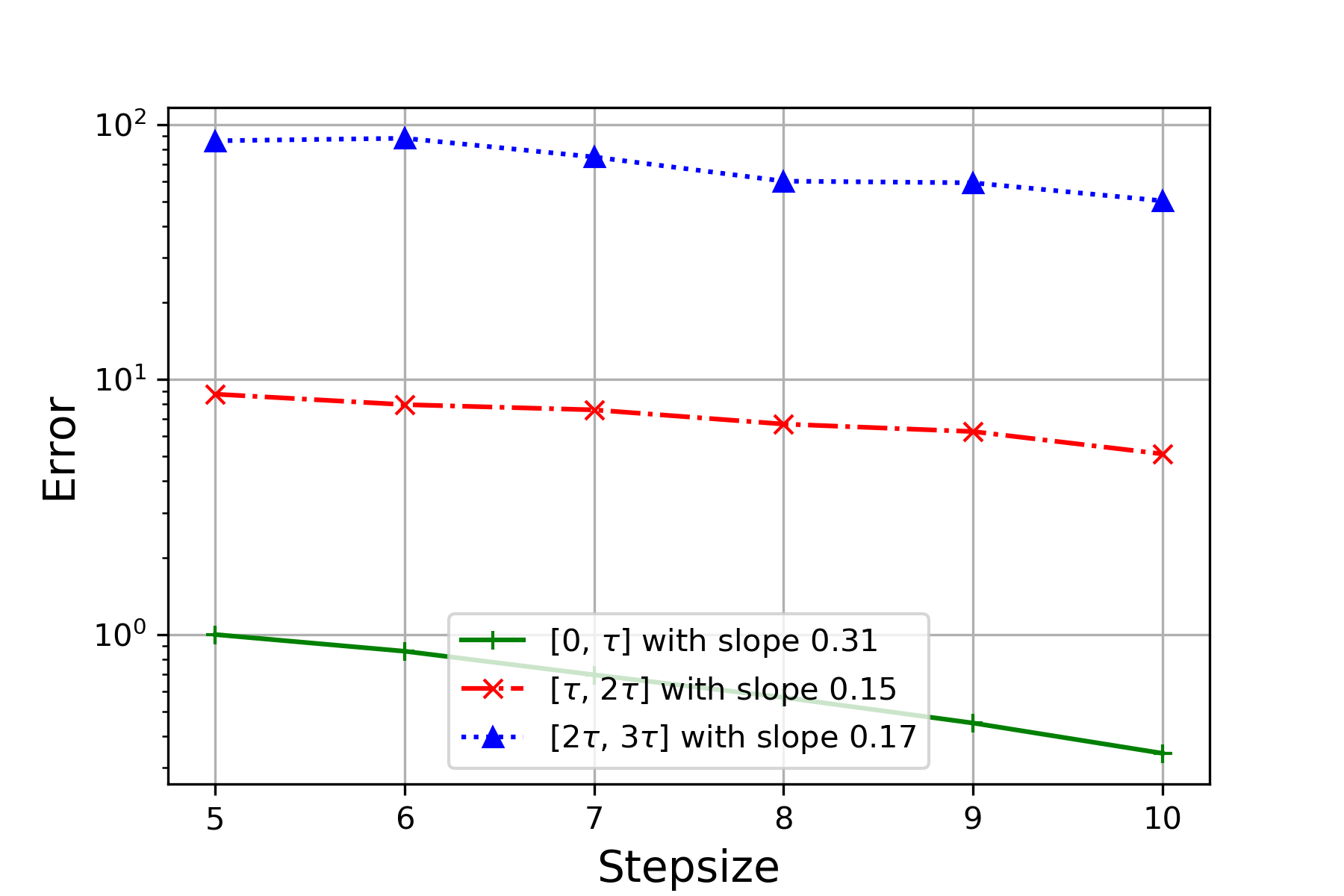}
    \caption{$\alpha_1=1$ and $\alpha_2=0.1$. \label{fig:alpha10.1,gamma50.5,f2}}
\end{subfigure}\\
\begin{subfigure}{0.4\textwidth}
    \centering
    \includegraphics[width=\textwidth]{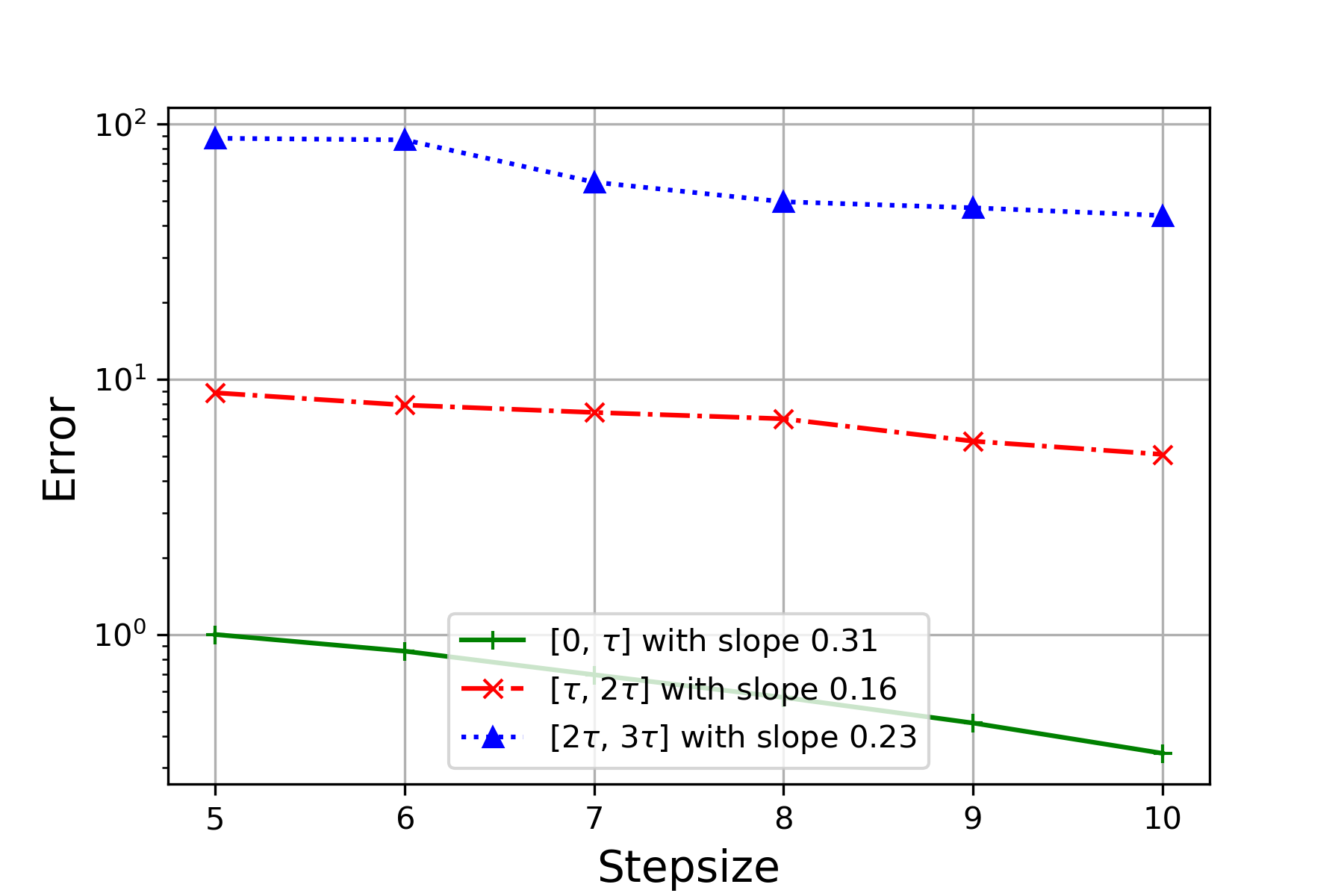}
    \caption{$\alpha_1=0.1$ and $\alpha_2=1$. \label{fig:alpha0.11,gamma50.5,f2}}  
\end{subfigure}
\begin{subfigure}{0.4\textwidth}
    \centering
    \includegraphics[width=\textwidth]{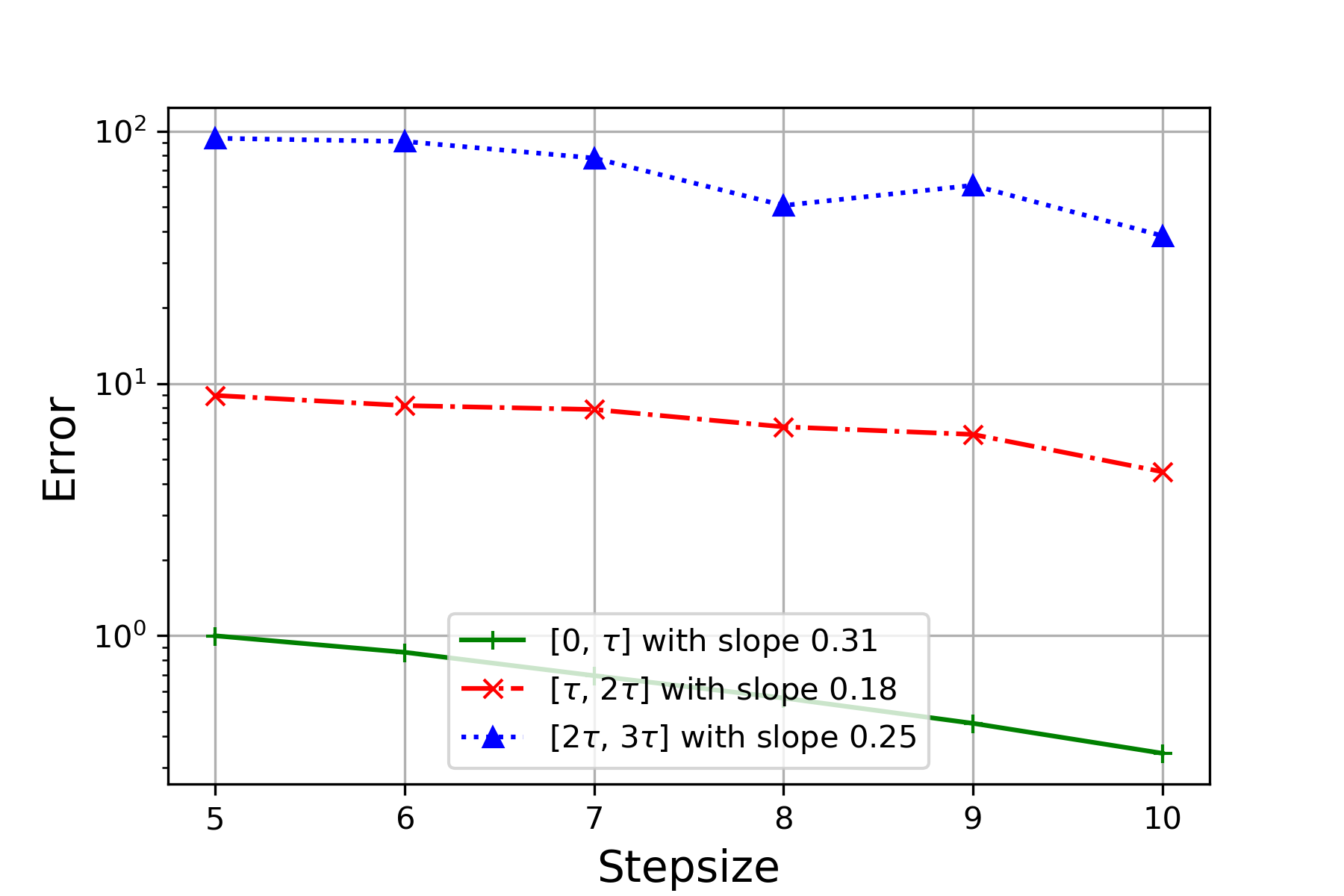}
    \caption{$\alpha_1=0.5$ and $\alpha_2=0.5$. \label{fig:alpha0.50.5,gamma50.5,f2}}  
\end{subfigure}
\caption{Mean square errors slope for $\gamma_1=5$ and $\gamma_2=0.5$ and values of $(\alpha_1,\alpha_2)=(0.1,0.1),(1,0.1),(0.1,1),(0.5,0.5)$ for \eqref{eq:f1} plus \eqref{eq:f2}. \label{fig:mse_g1f1_gamma50.5}}
\end{figure}

\begin{figure}
\centering
\begin{subfigure}{0.40\textwidth}
    \centering
    \includegraphics[width=\textwidth]{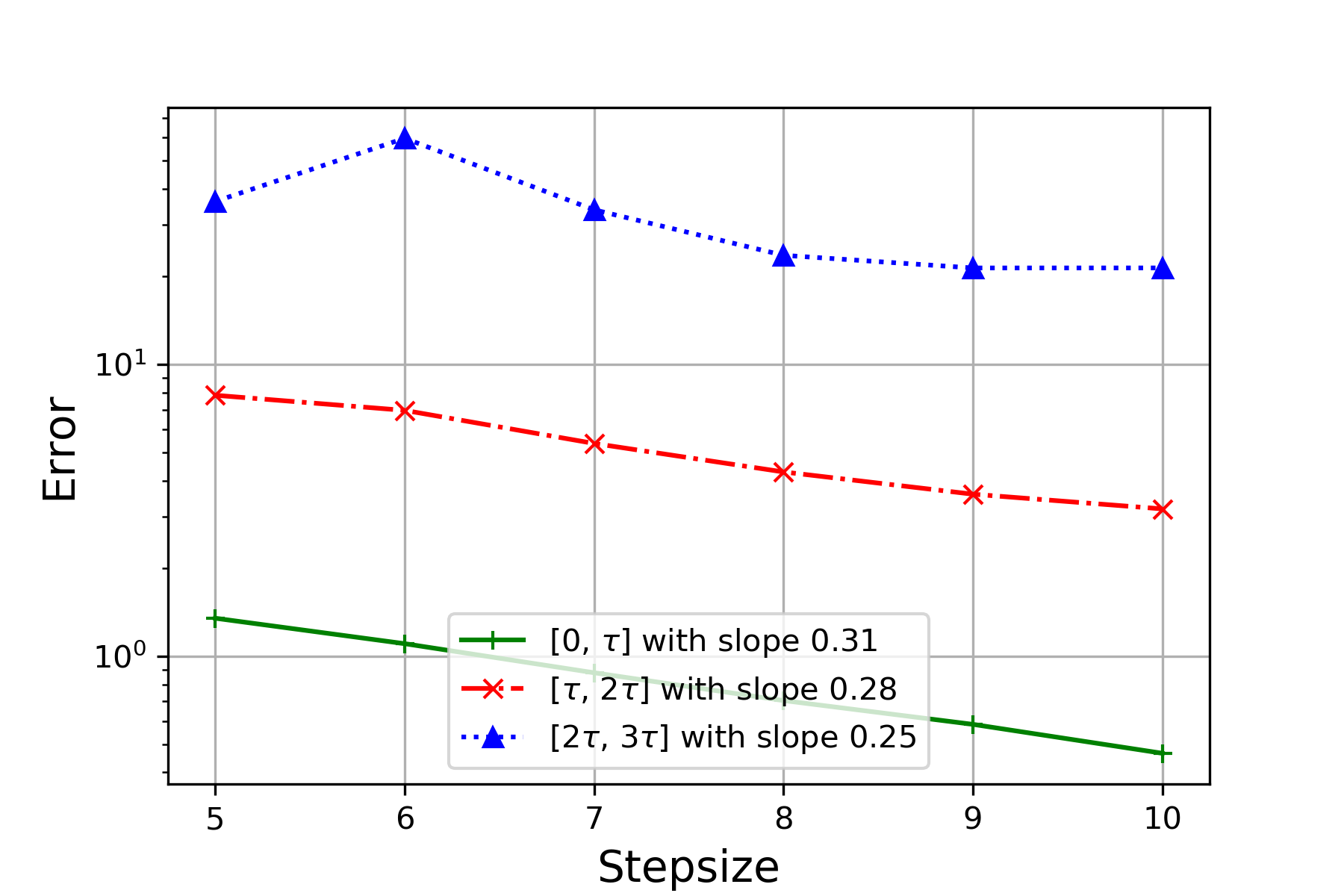}
    \caption{$\alpha_1=0.1$ and $\alpha_2=0.1$. \label{fig:alpha0.10.1,gamma50.1,f2}}
\end{subfigure}
\begin{subfigure}{0.40\textwidth}
    \centering
    \includegraphics[width=\textwidth]{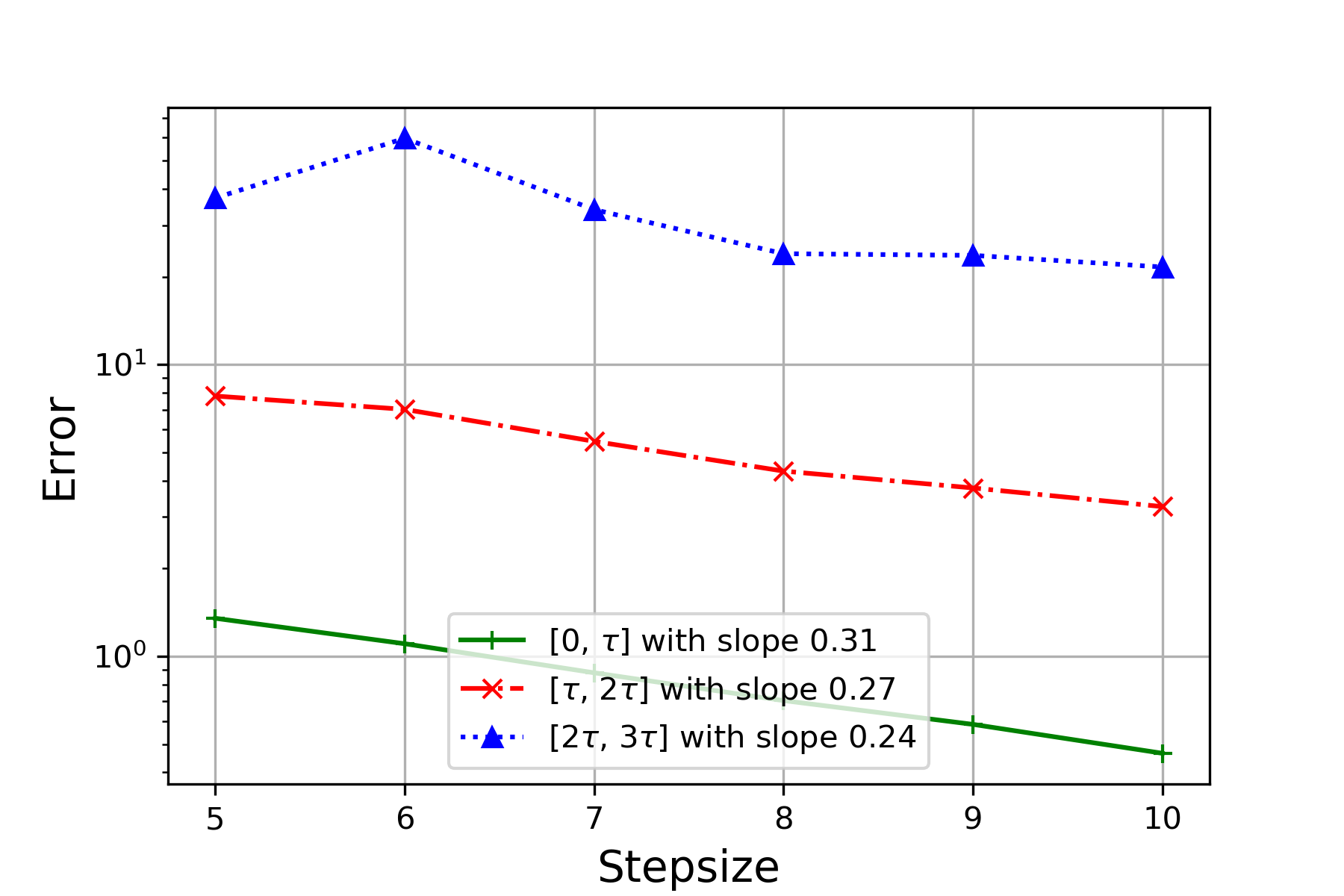}
    \caption{$\alpha_1=1$ and $\alpha_2=0.1$. \label{fig:alpha10.1,gamma50.1,f2}}
\end{subfigure}\\
\begin{subfigure}{0.4\textwidth}
    \centering
    \includegraphics[width=\textwidth]{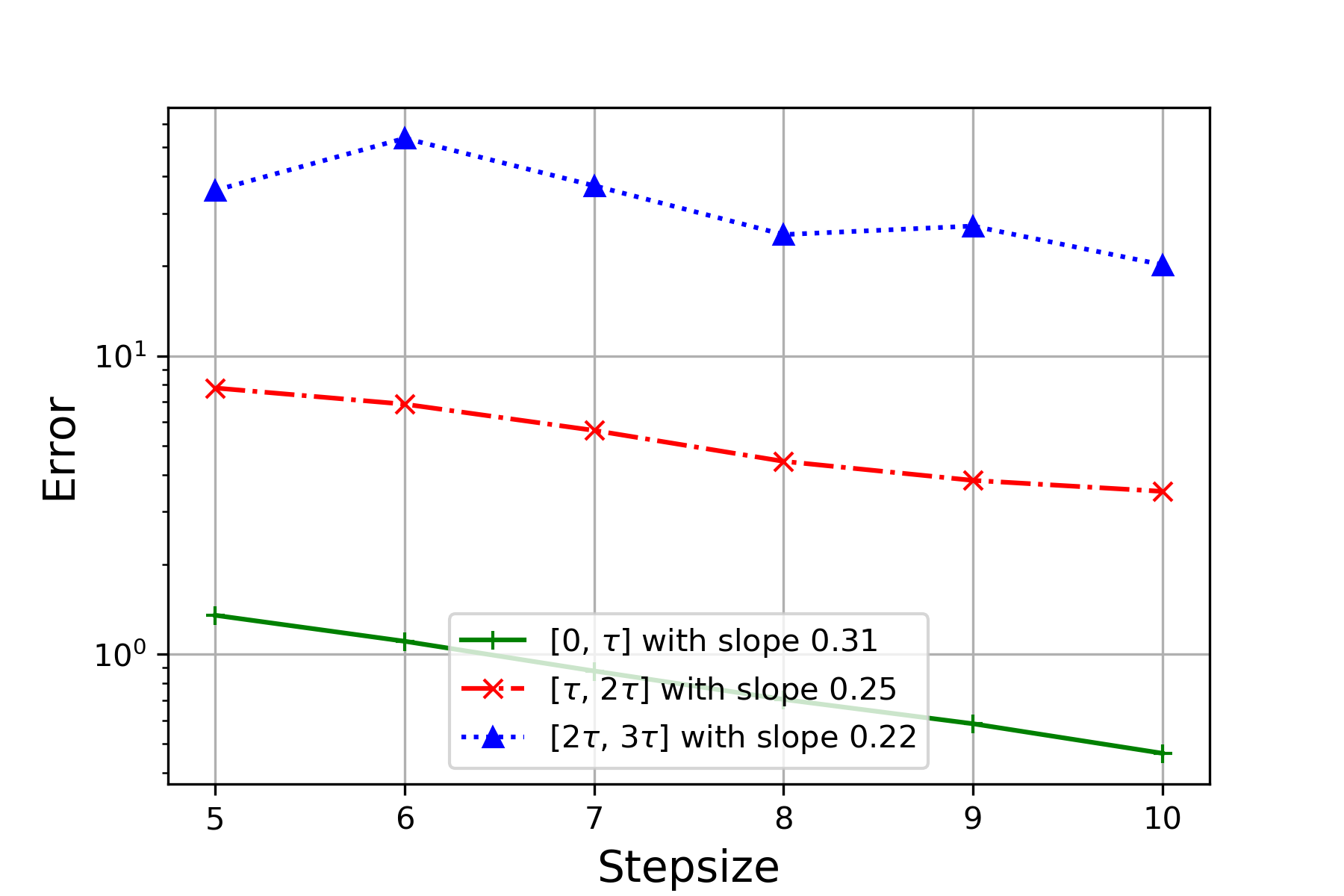}
    \caption{$\alpha_1=0.1$ and $\alpha_2=1$. \label{fig:alpha0.11,gamma50.1,f2}}  
\end{subfigure}
\begin{subfigure}{0.4\textwidth}
    \centering
    \includegraphics[width=\textwidth]{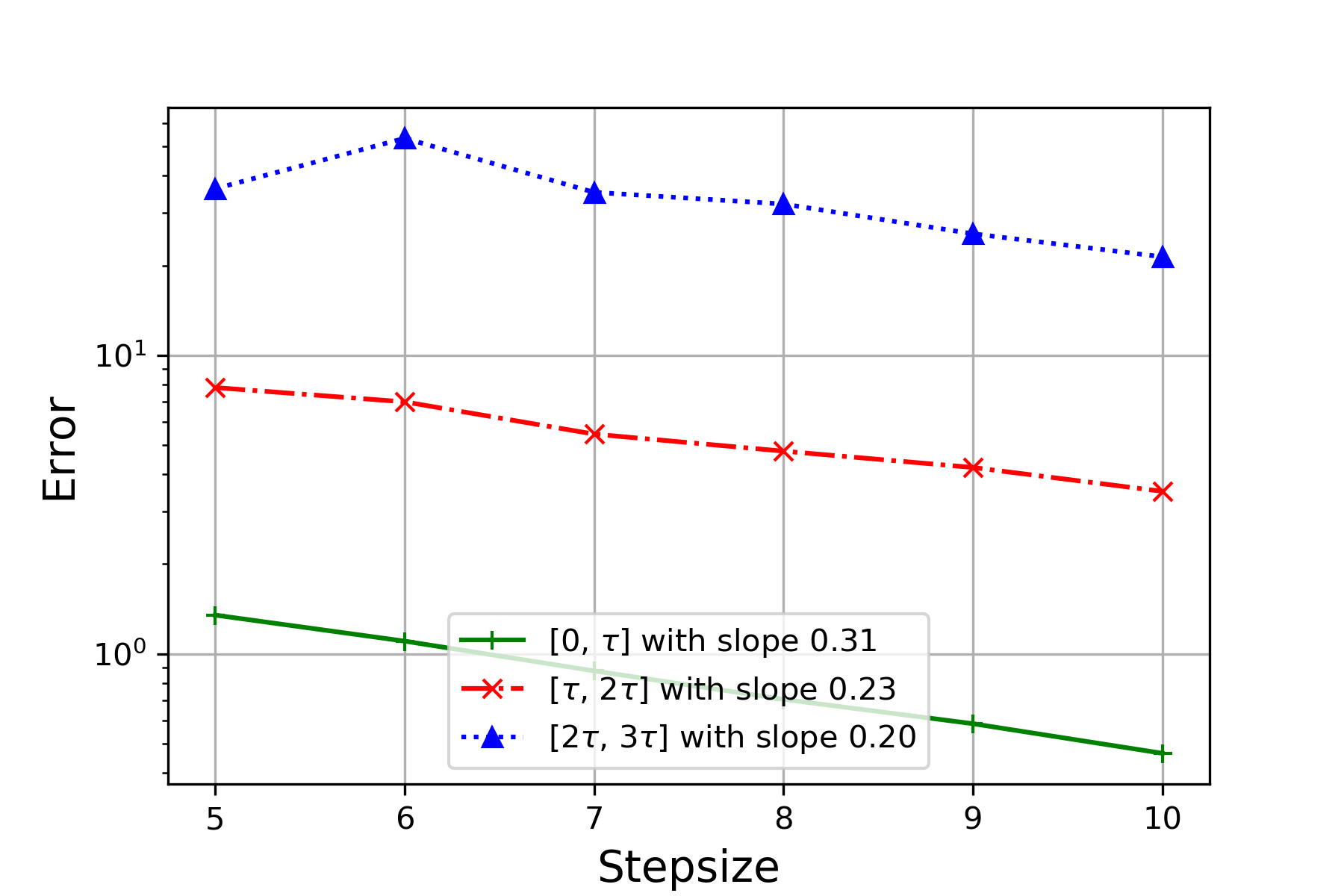}
    \caption{$\alpha_1=0.5$ and $\alpha_2=0.5$. \label{fig:alpha0.50.5,gamma50.1,f2}}  
\end{subfigure}
\caption{Mean square errors slope for $\gamma_1=5$ and $\gamma_2=1$ and values of $(\alpha_1,\alpha_2)=(0.1,0.1),(1,0.1),(0.1,1),(0.5,0.5)$ for \eqref{eq:f1} plus \eqref{eq:f2}. \label{fig:mse_g1f1_gamma50.1}}
\end{figure}
It can be observed that for each fixed $(\alpha_1,\alpha_2)$ pair and for each fixed $[j\tau, (j+1)\tau]$ interval, the order of convergence increases with $\gamma_2$; for each fixed $\gamma_2$ and for each fixed $[j\tau, (j+1)\tau]$ interval, the negative slops for $(\alpha_1,\alpha_2)=(0.1,0.1),(1,0.1),(0.1,1)$ are almost the same, which are slightly less than the negative slop of $(\alpha_1,\alpha_2)=(0.5,0.5)$; for each fixed $(\alpha_1,\alpha_2)$ pair and for each fixed $\gamma_2$, the negative slope decreases with $j$. All these observations coincide with Theorem \ref{rate_of_conv_expl_Eul}.
\end{exm}
\begin{exm} [Multiplicative noise]
\normalfont
In the following numerical tests we use \eqref{eq:f1} with \eqref{eqn:k2} and \eqref{eq:f2}. We fix the number of experiments $K=1000$ for each $N=2^l$, $l=5,\ldots,10$, and the reference solution is computed with stepsize $2^{-17}$; also, the horizon parameter is $n=3$. We vary $\gamma_2$, $\alpha_1$, but allow the same value of $\alpha_2$ as $\alpha_1$, ie, $\alpha_1=\alpha_2=\alpha$. 

We get the following results for $\gamma_2=0.1$: \\
letting $\alpha=0.1$, the negative mean square error slopes are
$0.34$, $0.31$, and $0.25$. See Figure \ref{fig:2alpha0.10.1,gamma51,f2}; \\
letting $\alpha=0.5$, the negative mean square error slopes are
$0.35$, $0.25$, and $0.23$. See Figure \ref{fig:2alpha0.50.5,gamma51,f2}; \\
letting $\alpha=1$, the negative mean square error slopes are
$0.35$, $0.23$, and $0.23$. See Figure \ref{fig:2alpha11,gamma50.1,f2}; \\
while, for $\gamma_2=0.5$: \\
letting $\alpha=0.1$, the negative mean square error slopes are $0.30$, $0.27$, and $0.27$. See Figure \ref{fig:2alpha0.10.1,gamma50.5,f2}; \\
letting $\alpha=0.5$, the negative mean square error slopes are $0.31$, $0.27$, and $0.24$. See Figure \ref{fig:2alpha0.50.5,gamma50.5,f2}; \\
letting $\alpha=1$, the negative mean square error slopes are
$0.34$, $0.27$, and $0.26$. See Figure \ref{fig:2alpha11,gamma50.5,f2}; \\
while, for $\gamma_2=1$: \\
letting $\alpha=0.1$, the negative mean square error slopes are $0.29$, $0.30$, and $0.23$. See Figure \ref{fig:2alpha0.10.1,gamma50.1,f2}; \\
letting $\alpha=0.5$, the negative mean square error slopes are $0.32$, $0.27$, and $0.25$. See Figure \ref{fig:2alpha0.50.5,gamma50.1,f2};\\
letting $\alpha=1$, the negative mean square error slopes are
$0.35$, $0.30$, and $0.28$. See Figure \ref{fig:2alpha11,gamma51,f2}.
\begin{figure}
\centering
\begin{subfigure}{0.31\textwidth}
    \centering
    \includegraphics[width=\textwidth]{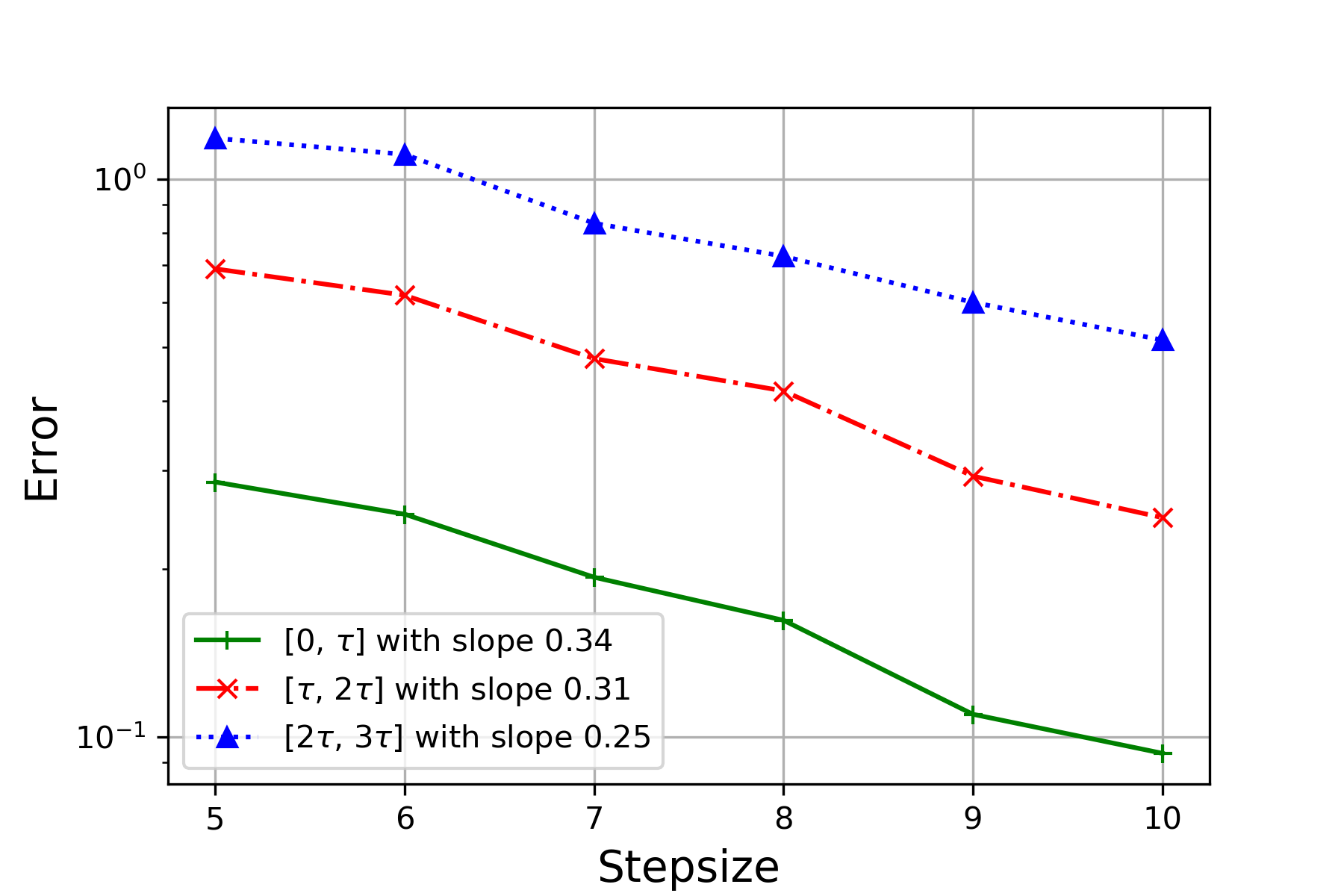}
    \caption{$\gamma_2=0.1$,  and $\alpha=0.1$. \label{fig:2alpha0.10.1,gamma51,f2}}
\end{subfigure}
\begin{subfigure}{0.31\textwidth}
    \centering
    \includegraphics[width=\textwidth]{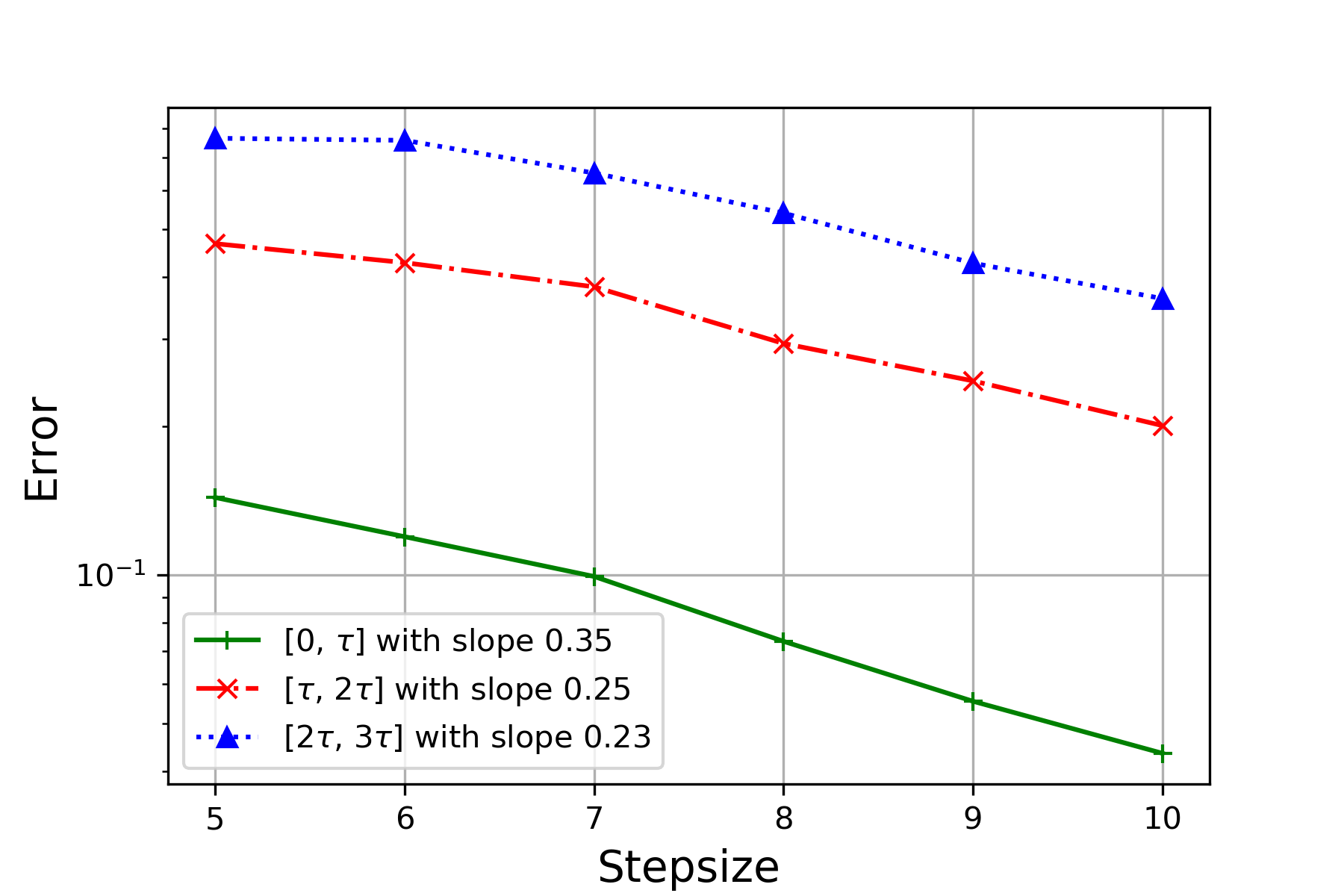}
    \caption{$\gamma_2=0.1$  and $\alpha=0.5$. \label{fig:2alpha0.50.5,gamma51,f2}}  
\end{subfigure} 
\begin{subfigure}{0.31\textwidth}
    \centering
    \includegraphics[width=\textwidth]{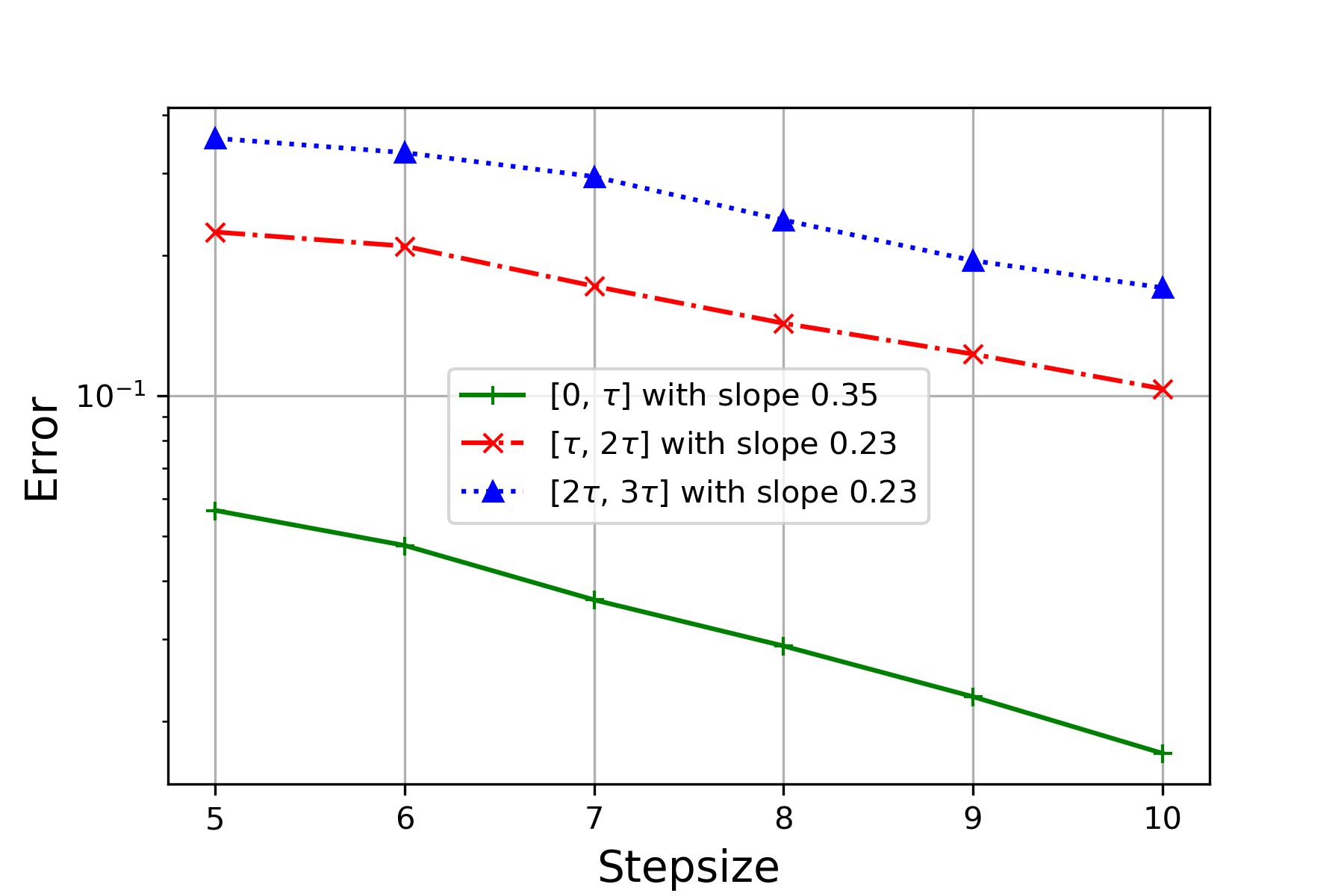}
    \caption{$\gamma_2=0.1$ and  $\alpha=1$. \label{fig:2alpha11,gamma50.1,f2}}  
\end{subfigure} \\
\begin{subfigure}{0.31\textwidth}
    \centering
    \includegraphics[width=\textwidth]{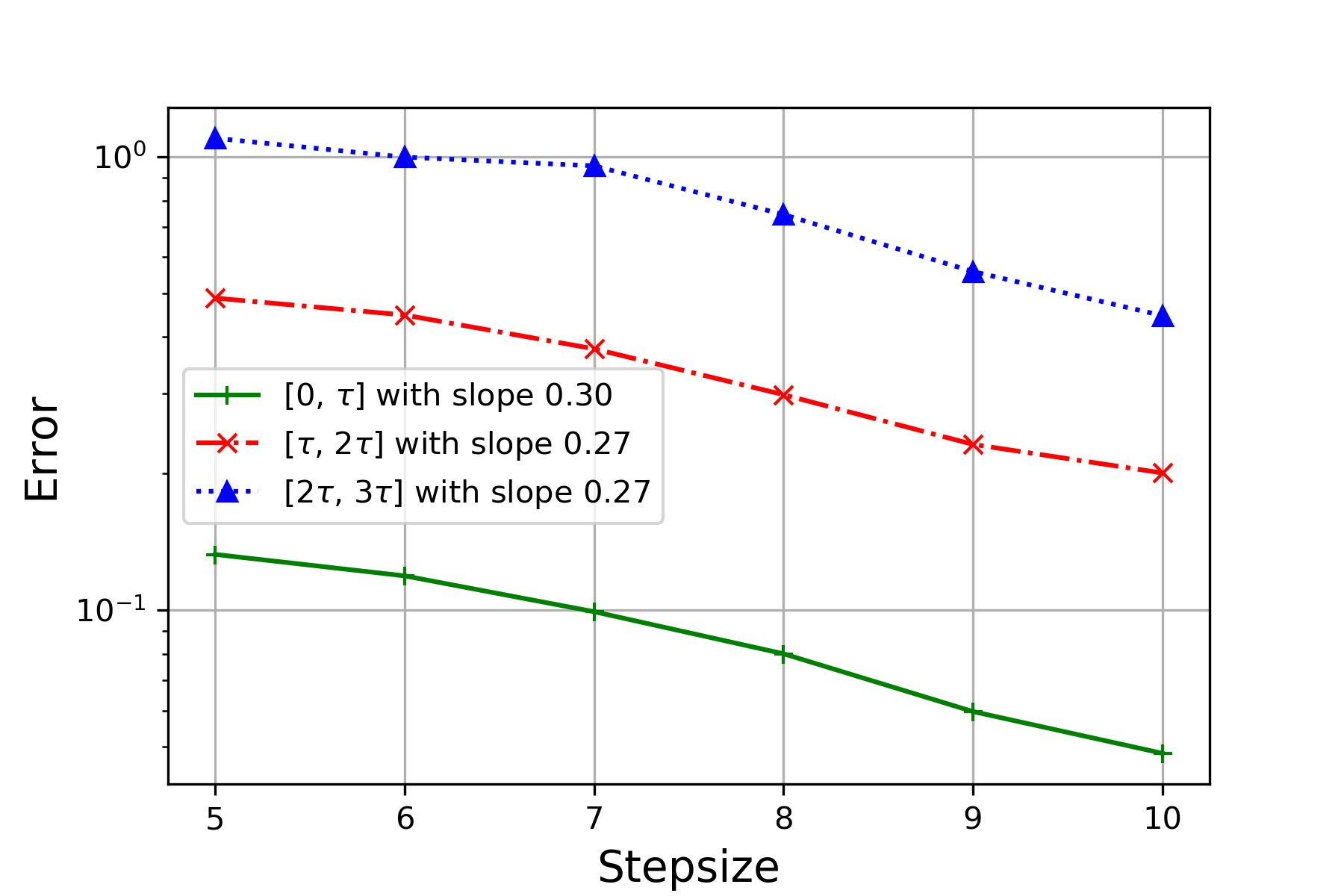}
    \caption{$\gamma_2=0.5$,  and $\alpha=0.1$. \label{fig:2alpha0.10.1,gamma50.5,f2}}
\end{subfigure} 
\begin{subfigure}{0.31\textwidth}
    \includegraphics[width=\textwidth]{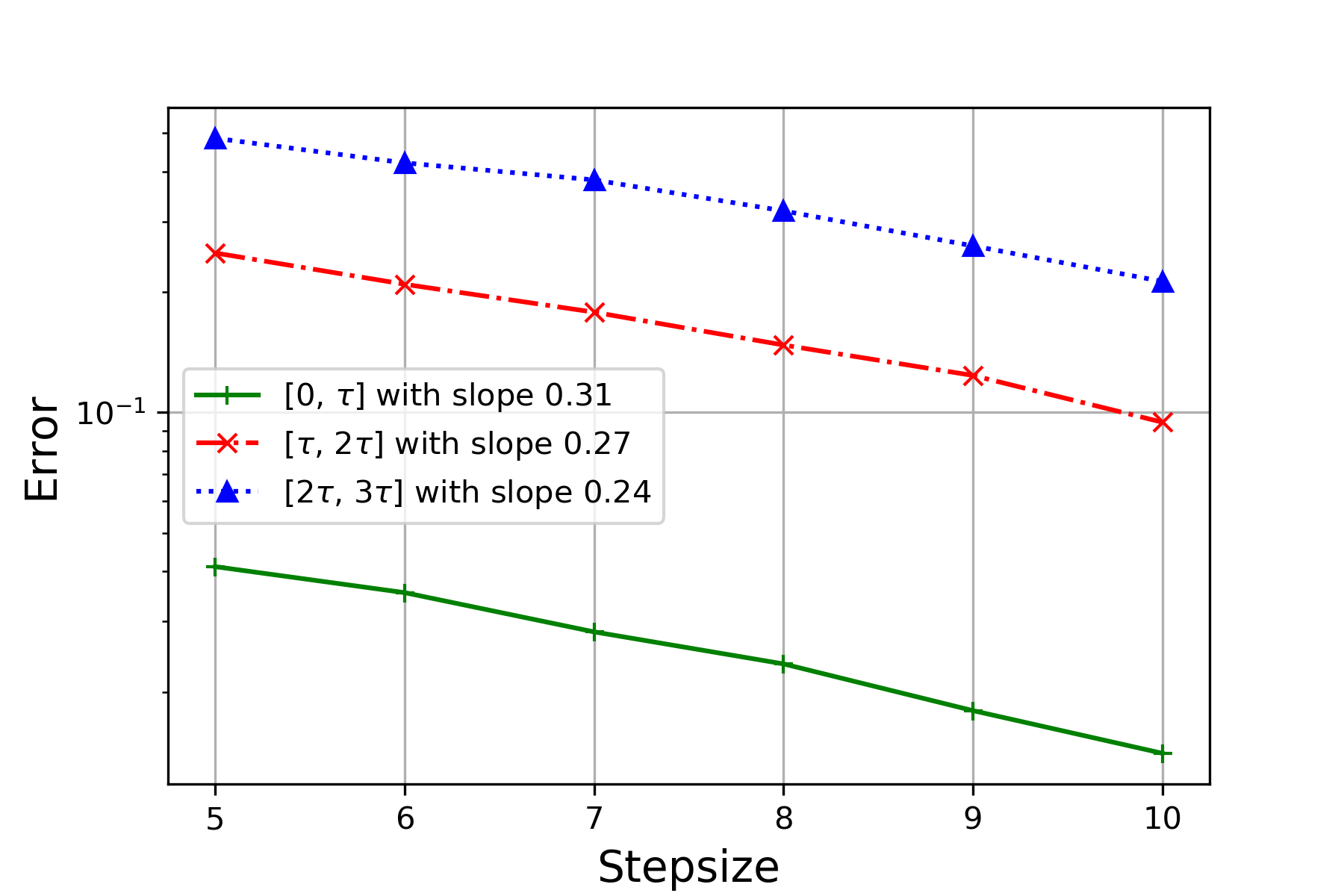}
    \caption{$\gamma_2=0.5$ and  $\alpha=0.5$. \label{fig:2alpha0.50.5,gamma50.5,f2}}  
\end{subfigure} 
\begin{subfigure}{0.31\textwidth}
    \centering
    \includegraphics[width=\textwidth]{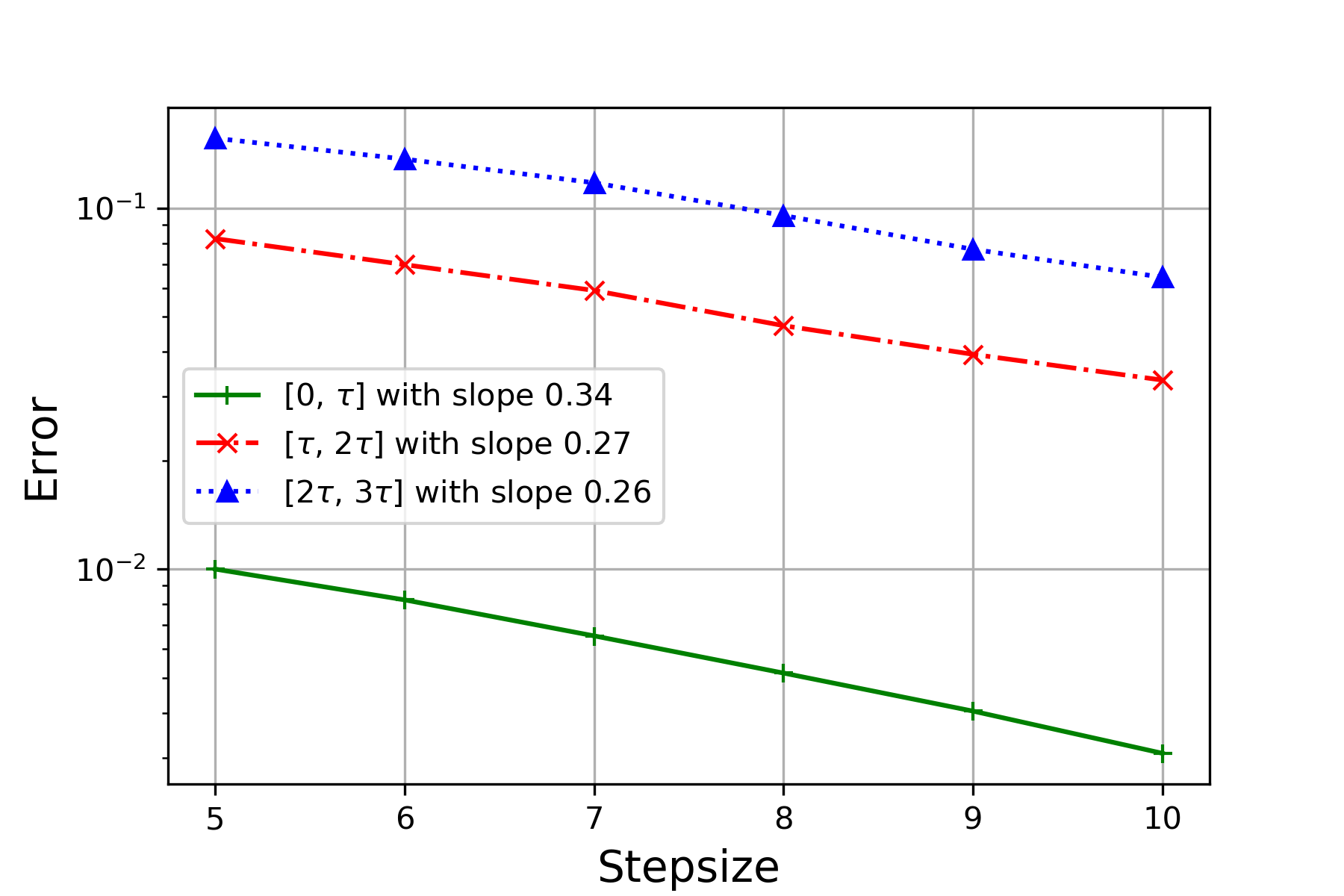}
    \caption{$\gamma_2=0.5$  and  $\alpha=1$. \label{fig:2alpha11,gamma50.5,f2}}  
\end{subfigure}
\\
\begin{subfigure}{0.31\textwidth}
    \centering
    \includegraphics[width=\textwidth]{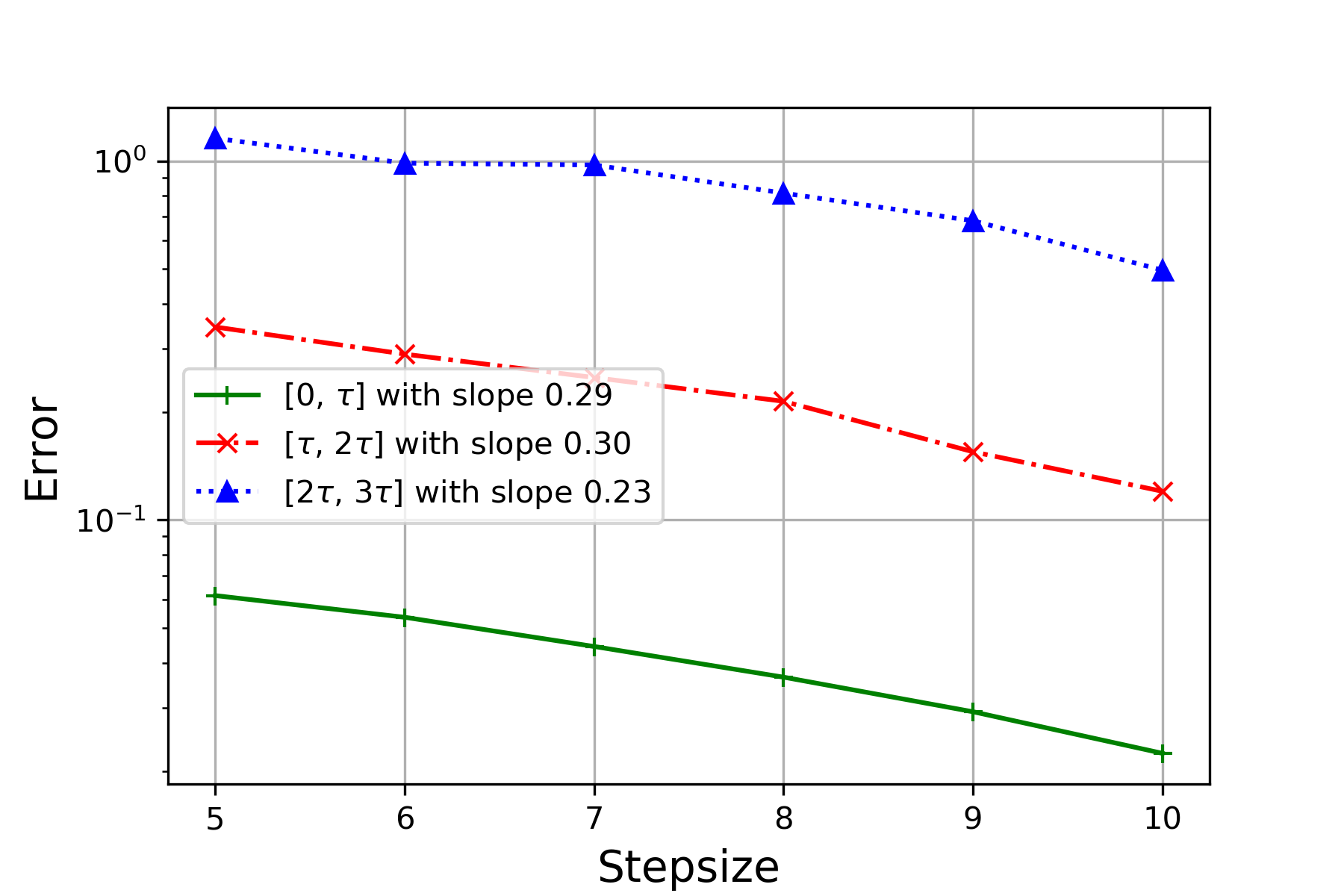}
    \caption{$\gamma_2=1$ and  $\alpha=0.1$. \label{fig:2alpha0.10.1,gamma50.1,f2}}
\end{subfigure} \begin{subfigure}{0.31\textwidth}
    \centering
    \includegraphics[width=\textwidth]{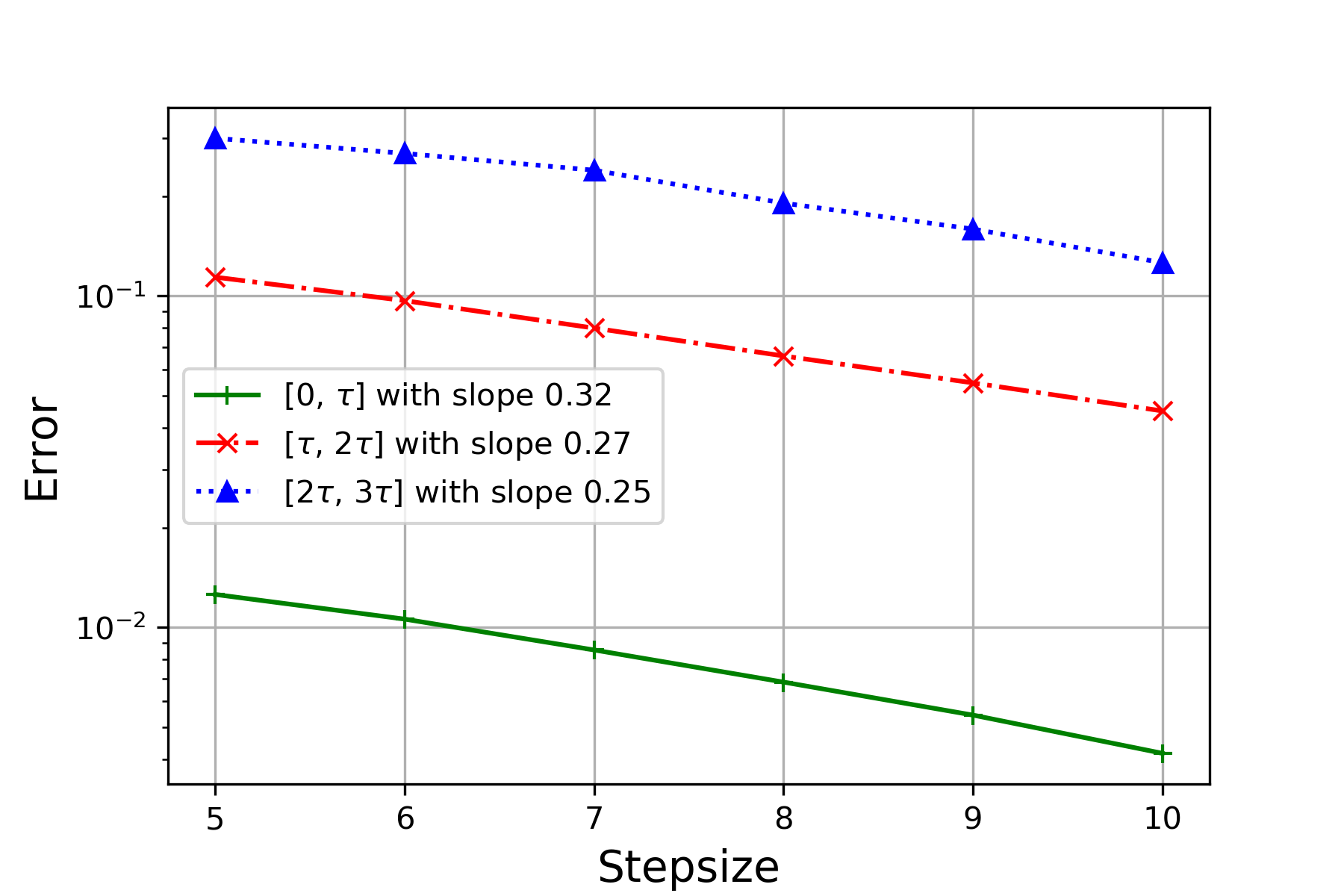}
    \caption{$\gamma_2=1$ and $\alpha=0.5$. \label{fig:2alpha0.50.5,gamma50.1,f2}}  
\end{subfigure}
\begin{subfigure}{0.31\textwidth}
    \centering
    \includegraphics[width=\textwidth]{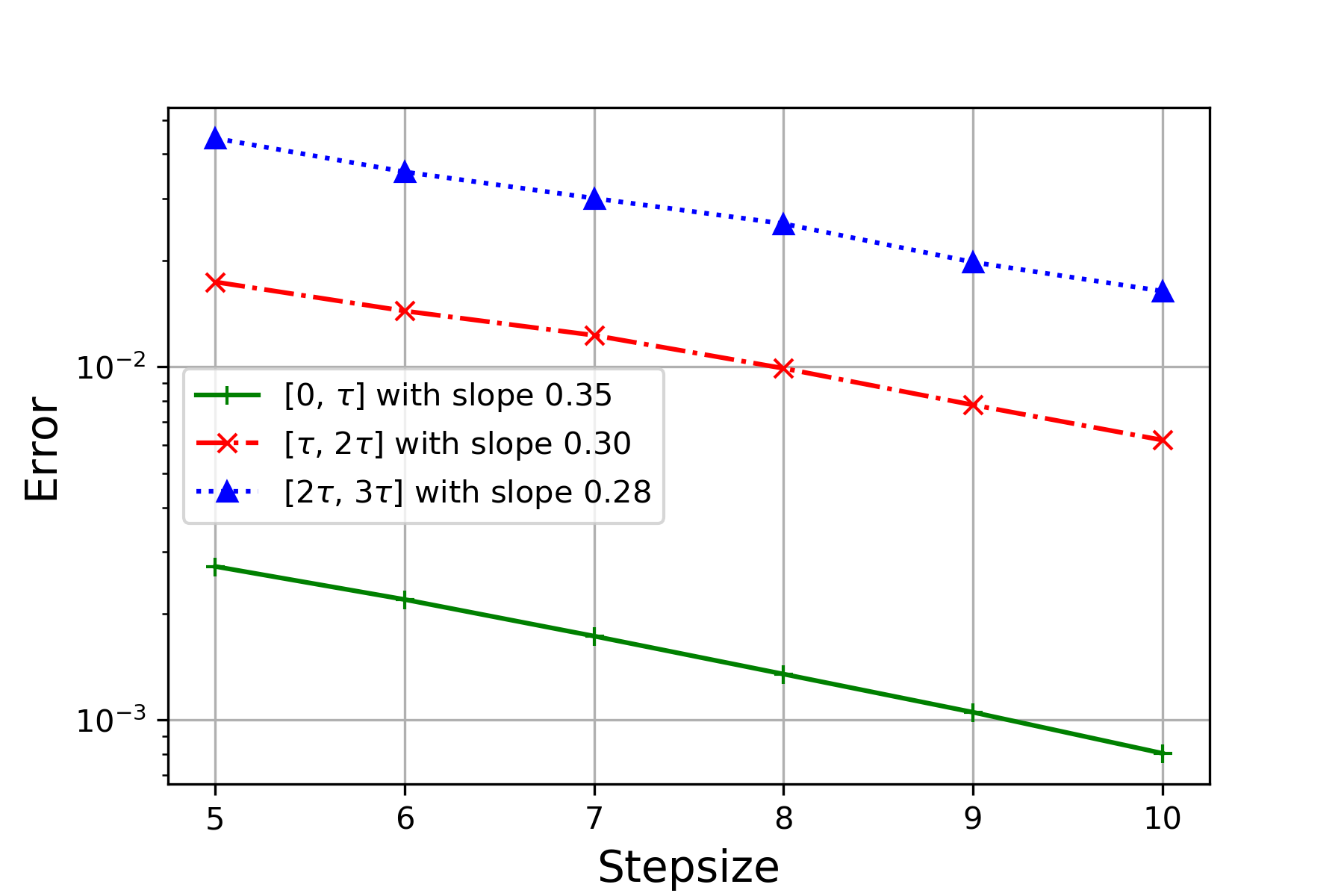}
    \caption{$\gamma_2=1$ and $\alpha=1$. \label{fig:2alpha11,gamma51,f2}}  
\end{subfigure}
\\
\caption{Mean square errors slope for \eqref{eq:f1} with \eqref{eqn:k2} plus \eqref{eq:f2} at $\gamma_2=0.1,0.5,1$ and values of $\alpha_1=\alpha_2=\alpha=0.1, 0.5, 1$. \label{fig:2mse_g2f2_gamma51}}
\end{figure}
One can compare the results horizontally or vertically in Figure \ref{fig:2mse_g2f2_gamma51}. For each fixed $(\alpha_1,\alpha_2)$ pair and for each fixed $\gamma_2$, the negative slope in general decreases with $j$. Horizontally, each fixed $\gamma_2$ and for each fixed $[j\tau, (j+1)\tau]$ interval, the error decreases significantly with increasing $\alpha_1$. Vertically, for each fixed $(\alpha_1,\alpha_2)$ pair and for each fixed $[j\tau, (j+1)\tau]$ interval, the numerical error decreases with $\gamma$. All these observations coincide with Theorem \ref{rate_of_conv_expl_Eul}.
\end{exm}

%%%%%%%%%%%%%%%%%%%%%%%%
\section{Appendix}
%%%%%%%%%%

\begin{lem}
\label{meas_Xgamma}
    Let $Y=(Y(t))_{t\geq 0}$ is $(\Sigma_t)_{t\geq 0}$-progressively measurable stochastic process and let $\xi:\Omega\to [a,b]$, $-\infty<a<b<+\infty$, is a random variable on $(\Omega,\Sigma,\mathbb{P})$. Then $Y(\xi)$ is $\sigma(\xi)\vee \Sigma_{b}$-measurable.
\end{lem}
\begin{proof}
Since $Y$ is progressively measurable we have that the mapping $[a,b]\times\Omega\ni (t,\omega)\to Y(t,\omega)\in\mathbb{R}^d$ is $\mathcal{B}([a,b])\otimes\Sigma_b$-to-$\mathcal{B}(\mathbb{R}^d)$ measurable. Let us define 
\begin{equation}
    \Omega\ni\omega\to H(\omega)=(\xi(\omega),\omega),
\end{equation}
and $\mathcal{F}_0=\{B\times F \ | \ B\in\mathcal{B}([a,b]), \ F\in\Sigma_b\}$. Note that $\mathcal{B}([a,b])\otimes\Sigma_b=\sigma(\mathcal{F}_0)$. Since for any $B\times F\in\mathcal{F}_0$ we have that $H^{-1}(B\times F)=\xi^{-1}(B)\cap F\in \sigma(\xi)\vee\Sigma_b$, we get by Proposition 2.3., page 6 in \cite{EC} that the function $H$ is $\sigma(\xi)\vee\Sigma_b$-to-$\mathcal{B}([a,b])\otimes\Sigma_b$ measurable. Since 
\begin{equation}
    \Omega\ni\omega\to Y(\xi(\omega),\omega)=(Y\circ H)(\omega)\in\mathbb{R}^d,
\end{equation}
we get for any $B\in\mathcal{B}(\mathbb{R}^d)$ that $(Y\circ H)^{-1}(B)=H^{-1}(Y^{-1}(B))\in\sigma(\xi)\vee\Sigma_b$. This implies the thesis.
\end{proof}
%%%%%%%%%%%%%%%%%%%%%%%%
\bibliographystyle{plain}
\bibliography{biblioDDDE.bib}
\end{document}